\documentclass[12pt]{article}

\usepackage{booktabs}
\usepackage{graphicx}
\usepackage{amsmath}
\usepackage{amssymb}
\usepackage{latexsym}
\usepackage{subfigure}
\usepackage{crop}
\usepackage{framed}
\usepackage{algorithm}
\usepackage[noend]{algpseudocode}
\usepackage{multirow}
\usepackage{bm}
\usepackage{bbm}
\usepackage{enumerate}
\usepackage{framed} 
\usepackage{url}
\usepackage{array}
\usepackage{paralist}
\usepackage{diagbox}
\usepackage[colorlinks = true, pdfstartview = FitV, linkcolor = blue, citecolor = blue, urlcolor = blue]{hyperref}

\usepackage{amsthm}
\usepackage[many]{tcolorbox}
\theoremstyle{plain}
\newtheorem{assumption}{Assumption}
\newtheorem{condition}{Condition}
\newtheorem{corollary}{Corollary}
\newtheorem{lemma}{Lemma}
\newtheorem{theorem}{Theorem}
\tcolorboxenvironment{assumption}{sharp corners,boxrule=0.4pt,colback=white,before skip=\topsep,after skip=\topsep}
\tcolorboxenvironment{condition}{sharp corners,boxrule=0.4pt,colback=white,before skip=\topsep,after skip=\topsep}
\tcolorboxenvironment{corollary}{sharp corners,boxrule=0.4pt,colback=white,before skip=\topsep,after skip=\topsep}
\tcolorboxenvironment{lemma}{sharp corners,boxrule=0.4pt,colback=white,before skip=\topsep,after skip=\topsep}
\tcolorboxenvironment{theorem}{sharp corners,boxrule=0.4pt,colback=white,before skip=\topsep,after skip=\topsep}

\usepackage{xspace}

\renewcommand\th{\textsuperscript{th}\xspace}

\usepackage{fullpage}
\usepackage{multirow}

\usepackage[capitalise]{cleveref}
\crefname{equation}{}{}
\creflabelformat{equation}{\textup{(#2#1#3)}}

\usepackage[sort,nocompress]{cite}

\usepackage{arydshln}
\usepackage{enumitem}
\setlist[enumerate,1]{leftmargin=*,wide=0em, label = {\bfseries \arabic*.}}
\setlist[itemize,1]{leftmargin=*,wide=0em}

\usepackage{pifont}
\newcommand{\reals}{\mathbb{R}}

\renewcommand {\AA}  { {\mathbf{A}} }

\newcommand {\HH}  { {\mathbf{H}} }

\newcommand {\UU}  { {\mathbf{U}} }
\newcommand {\PP}  { {\mathbf{P}} }

\newcommand{\eye}{\mathbf{I}}

\newcommand {\zz}  { {\bf z} }
\newcommand {\bgg}  { {\bf g} }

\newcommand {\xx}  { {\bf x} }

\newcommand {\yy}  { {\bf y} }

\newcommand {\pp}  { {\bf p} }

\newcommand {\vv}  { {\bf v} }
\newcommand {\ww}  { {\bf w} }

\newcommand {\bb}  { {\bf b} }

\newcommand {\ee}  { {\bf e} }
\newcommand {\zero}  { {\bf 0} }

\renewcommand{\vec}[1]{\ensuremath{\mathbf{#1}}}
\newcommand{\grad}{\ensuremath {\vec \nabla}}

\newcommand{\defeq}{\triangleq}

\definecolor{forestgreen}{rgb}{0.13, 0.55, 0.13}

\newcounter{comment}

\newenvironment{remark}[1][Remark:]{\begin{trivlist}
		\item[\hskip \labelsep {\bfseries #1}]}{\end{trivlist}}

\usepackage{listings} 

\definecolor{mygreen}{rgb}{0,0.6,0}
\definecolor{mygray}{rgb}{0.5,0.5,0.5}
\definecolor{mymauve}{rgb}{0.58,0,0.82}

\newcommand*\norm[1]{\left\| #1\right\|}

\newcommand\bbR{\ensuremath{\mathbb{R}}} 

\DeclareMathOperator*{\argmin}{arg\,min}
\usepackage{titlesec}
\titleclass{\subsubsubsection}{straight}[\subsection]

\usepackage{booktabs}

\newcounter{subsubsubsection}[subsubsection]
\renewcommand\thesubsubsubsection{\thesubsubsection.\arabic{subsubsubsection}}

\titleformat{\subsubsubsection}
  {\normalfont\normalsize}{\thesubsubsubsection}{1em}{}
\titlespacing*{\subsubsubsection}{0pt}{3.25ex plus 1ex minus .2ex}{1.5ex plus .2ex}

\begin{document}

\title{DINGO: Distributed Newton-Type Method for Gradient-Norm Optimization}
\author{
    Rixon Crane\footnote{School of Mathematics and Physics, University of Queensland, Australia. Email: r.crane@uq.edu.au} 
    \qquad 
    Fred Roosta\footnote{School of Mathematics and Physics, University of Queensland, Australia, and International Computer Science Institute, Berkeley, USA. Email: fred.roosta@uq.edu.au}
    }
\maketitle

\begin{abstract}
    For optimization of a sum of functions in a distributed computing environment, we present a novel communication efficient Newton-type algorithm that enjoys a variety of advantages over similar existing methods. Similar to Newton-MR, our algorithm, DINGO, is derived by optimization of the gradient's norm as a surrogate function. DINGO does not impose any specific form on the underlying functions, and its application range extends far beyond convexity. In addition, the distribution of the data across the computing environment can be arbitrary. Further, the underlying sub-problems of DINGO are simple linear least-squares, for which a plethora of efficient algorithms exist. Lastly, DINGO involves a few hyper-parameters that are easy to tune. Moreover, we theoretically show that DINGO is not sensitive to the choice of its hyper-parameters in that a strict reduction in the gradient norm is guaranteed, regardless of the selected hyper-parameters. We demonstrate empirical evidence of the effectiveness, stability and versatility of our method compared to other relevant algorithms.
\end{abstract}

\section{Introduction}\label{Section: Introduction}
	Consider the optimization problem
	\begin{align}
	\label{eq: finite sum}
	\min_{\ww \in \bbR^{d}} \bigg\{ f(\ww) \triangleq \frac{1}{m} \sum_{i = 1}^{m} f_{i}(\ww) \bigg\},
	\end{align}
	in a distributed computing environment involving $ m $ workers, in which the $ i\th $ worker can only locally access the $ i\th $ component function, $ f_{i} $. Such distributed computing settings arise increasingly more frequent as a result of technological and communication advancements that have enabled the collection of and access to large scale datasets.
    
    As a concrete example, take a data fitting application, in which given $ n $ data points, $ \{\xx_{i}\}_{i=1}^{n} $, and their corresponding loss function, $ \ell_{i}(\ww; \xx_{i}) $, parameterized by $ \ww $, the goal is to minimize the overall loss as $\min_{\ww \in \bbR^{d}} \; \sum_{i=1}^{n} \ell_i(\ww; \xx_{i})/n$. Here, each $\ell_{i}:\bbR^d\rightarrow\bbR $ corresponds to an observation (or a measurement) which models the loss (or misfit) given a particular choice of the underlying parameter $ \ww $. Such problems appear frequently in machine learning, e.g., \cite{shalev2014understanding,friedman2001elements,mohri2012foundations} and scientific computing, e.g., \cite{ye2017optimization,rodoas1,rodoas2}.
    However, in ``big data'' regimes where $ n \gg 1 $, lack of adequate computational resources, in particular storage, can severely limit, or even prevent, any attempts at solving such optimization problems in a traditional stand-alone way, e.g., using a single machine. 
    This can be remedied through \textit{distributed computing}, in which resources across a network of stand-alone computational nodes are ``pooled'' together so as to scale to the problem at hand. In such a setting, where $ n $ data points are distributed across $ m $ workers,  one can instead consider formulation \cref{eq: finite sum} with  
    \begin{align}
    \label{eq: finite sum dist fi}
    f_{i}(\ww) \triangleq \frac{1}{|S_{i}|}\sum_{j \in S_i} \ell_j(\ww; \xx_{j}), 
    \quad i = 1,\ldots, m,
    \end{align}
    where $S_{i} \subseteq \{1,\ldots,n\}$, with cardinality denoted by $ |S_{i}| $, correspond to the distribution of data across the nodes, i.e., the $ i\th $ node has access to a portion of the data of size $ |S_{i}| $, indexed by the set $ S_{i}  $.
    
    The increasing need for such computing environments, in general, has motivated the development of many distributed computing frameworks, such as MapReduce \cite{MapReduce}, Apache Spark \cite{Spark} and Alchemist \cite{Alchemist}. 
    Unfortunately, with their many benefits comes the caveat of \textit{communication costs}.  
    Indeed, in distributed settings, the amount of communications, i.e., messages
    exchanged across the network, are often considered a major bottleneck of computations (often more so than local computation times), as they can be expensive in terms of both physical resources and time through latency \cite{bekkerman2011scaling, DiSCO}. 
    
    In this light, when designing efficient distributed optimization algorithms, it is essential to keep communication overhead as low as possible. 
    First-order methods \cite{beck2017first}, e.g., stochastic gradient descent (SGD) \cite{SGD}, solely rely on gradient information and as a result are rather easy to implement in distributed settings.
    They often require the performance of many computationally inexpensive iterations, 
    which can be suitable for execution on a single machine.
    However, as a direct consequence, they can incur excessive communication costs in distributed environments 
    and, hence, they might not be able to take full advantage of the available distributed computational resources. 
    
    By employing curvature information in the form of the Hessian matrix, second-order methods aim at transforming the gradient such that it is a more suitable direction to follow.  
    Compared with first-order alternatives, although second-order methods perform more computations per iteration,  they often require far fewer iterations to achieve similar results. 
    In distributed settings, this feature can directly translate to significantly less communication costs. 
    As a result, distributed second-order methods have the potential to become the method of choice for distributed optimization tasks. 
    
    The rest of this paper is organized as follows. 
    After introducing the notation used throughout the paper, the remainder of \cref{Section: Introduction} discusses related works and, in their light, briefly summarizes our contributions. The distributed computing environment considered in this paper is also introduced towards the end of \cref{Section: Introduction}. 
    In \cref{Section: Our Algorithm}, we provide a description and derivation of our algorithm.
    Theoretical properties and assumptions are detailed in \cref{Section: Theoretical Analysis}.
    Results of numerical experiments, with comparisons to other relevant methods, are presented in \cref{Section: Experiments}. 
    Concluding remarks are made in \cref{Section: conclusion}.
    Proofs are provided in Appendix~\ref{appendix: proofs}.
    
    \subsubsection*{Notation}
    Throughout this paper, bold lowercase and bold uppercase letters, such as $\vv$ and $\AA$, denote column vectors and matrices, respectively.
    We let $\langle\cdot,\cdot\rangle$ denote the common \textit{Euclidean inner product} defined by 
    $\langle \xx,\yy \rangle = \xx^T\yy$ for $\xx,\yy\in\bbR^d$.
    Given a vector $\vv$ and matrix $\AA$, we denote their vector $\ell_2$ norm and matrix \textit{spectral} norm as 
    $\|\vv\|$ and $\|\AA\|$, respectively. 
    For $\xx,\zz\in\bbR^d$ we let $[\xx,\zz]\defeq\big\{\xx+\tau(\zz-\xx) \:|\: 0\leq\tau\leq1\big\}$.
    The \textit{range} and \textit{null space} of a matrix $\AA\in\bbR^{m\times n}$ are denoted by 
    $\mathcal{R}(\AA) = \{\AA\xx \:|\: \xx\in\bbR^n\}$ and $\mathcal{N}(\AA) = \{\xx\in\bbR^n \:|\: \AA\xx=\zero\}$, respectively. 
    The \textit{orthogonal complement} of a vector subspace $V$ of $\bbR^d$ is denoted by 
    $V^\perp = \big\{\xx\in\bbR^d \:|\: \langle\xx,\yy\rangle=0\text{ for all }\yy\in V\big\}$.
    The \textit{Moore--Penrose inverse} \cite{AGeneralizedInverseForMatrices} of $\AA$ is denoted by $\AA^\dagger$. 
    We let $\ww_t\in\mathbb{R}^{d}$ denote the point at iteration $t$. For notational convenience, we denote
    \begin{align}\label{eq: H and g}
    \bgg_{t,i} \defeq \grad f_i(\ww_t), \quad
    \HH_{t,i} \defeq \grad^2 f_i(\ww_t), \quad
    \bgg_{t} \defeq \grad f(\ww_t), \quad
    \HH_{t} \defeq \grad^2 f(\ww_t).
    \end{align}
    We also let
    \begin{equation}\label{eq: H tilde and g tilde}
    \tilde{\HH}_{t,i} 
    \defeq
    \begin{bmatrix}
    \HH_{t,i} \\
    \phi\eye
    \end{bmatrix} \in\bbR^{2d\times d} 
    \quad\text{and}\quad
    \tilde{\bgg}_t 
    \defeq
    \begin{pmatrix}
    \bgg_t \\
    \zero
    \end{pmatrix} \in\bbR^{2d},
    \end{equation}
    where $\phi>0$, $ \eye $ is the identity matrix, and $ \zero $ is the zero vector. Further notation will be introduced where necessary. 
    
    \subsubsection*{Related Work} 
    Owing to the above-mentioned potential, many distributed second-order optimization algorithms have recently emerged to solve \cref{eq: finite sum}. 
    Among them, most notably are GIANT \cite{GIANT}, DiSCO \cite{DiSCO}, DANE \cite{DANE}, InexactDANE and AIDE \cite{AIDE}. 
    However, most of these methods come with disadvantages that can limit their applicability. 
    For example, not only do many of these algorithms rely on rather stringent (strong) convexity assumptions, but also the underlying functions are required to be of a specific form. 
    Further, for some of these methods, the distribution of data is required to satisfy particular, and often rather restrictive, assumptions. 
    The underlying sub-problems for many of these methods involve non-linear optimization problems, which might not be easy to solve. 
    Last but not least, the success of many of these methods is tightly intertwined with fine-tuning of (often many) hyper-parameters, which is very expensive for some and downright impossible for others.
    
    More specifically, the theoretical analysis of both GIANT and DiSCO is limited to the case where each $ f_{i}  $  is strongly convex, and for GIANT they are also of the special form where in \cref{eq: finite sum dist fi} we have 
    $ \ell_j(\ww; \xx_j) = \psi_j\big(\langle\ww,\xx_j\rangle\big) + \gamma\|\ww\|^2$, $\gamma>0$ is a regularization parameter and $\psi_j:\bbR\rightarrow\bbR$ is convex, e.g., linear predictor models. 
    Moreover, in theory, GIANT requires $|S_{i}| > d$, for all $i$, which can be violated in many practical situations. 
    These two methods, compared to DANE, InexactDANE and AIDE, have a significant advantage in that they are 
    easier in implementation. More specifically, they involve fewer hyper-parameters to tune and their sub-problems involve solutions of symmetric positive-definite linear-systems, for which the conjugate gradient (CG) method \cite{nocedal2006numerical} is highly effective. 
    In sharp contrast, the performance of DANE, InexactDANE and AIDE is greatly affected by rather meticulous fine-tuning of many more hyper-parameters. Further, their underlying sub-problems involve non-linear optimization problems, which can themselves be rather difficult to solve. 
    However, the theoretical analysis of DANE, InexactDANE and AIDE extends beyond the simple strongly convex case and these methods are applicable to more general non-convex \eqref{eq: finite sum}.
    
    Another notable and recent distributed optimization algorithm is Newton-ADMM \cite{fang2018distributed}.
    Similar to both GIANT and DiSCO, Newton-ADMM requires $f_i$ to be of the special form where in \eqref{eq: finite sum dist fi} we have $\ell_j(\ww;\xx_j) = \psi_j(\ww;\xx_j) + g(\ww)$, with smooth convex $\psi_j:\bbR^d\rightarrow\bbR$ and smooth (strongly) convex regularization $g:\bbR^d\rightarrow\bbR$. 
    This algorithm uses Newton's method to solve the sub-problems of ADMM \cite{boyd2011distributed}, which allows it to effectively utilize GPUs. 
    Hence, like DANE, InexactDANE and AIDE, the sub-problems of Newton-ADMM are non-linear optimization problems. 
    However, unlike other Newton-type methods mentioned here, in Newton-ADMM the curvature is not employed in transforming the gradient direction. Instead, the second-order information is leveraged, through the use of Newton-type methods, for efficient solution of ADMM sub-problems. As a result, since Newton-ADMM is not second-order in its construction, it is not considered further in this paper.
    
    \subsubsection*{Contributions} 
    Here, we present a novel communication efficient distributed second-order optimization method that aims to alleviate many of the disadvantages of related methods.
    Our approach is inspired by and follows many ideas of recent results on Newton-MR \cite{Newton-MR}, which extends the application range of the classical Newton-CG beyond (strong) convexity and smoothness. In fact, our method can be viewed as a distributed variant of Newton-MR. 
    More specifically, our algorithm, named DINGO for ``\textbf{DI}stributed \textbf{N}ewton-type method for \textbf{G}radient-norm \textbf{O}ptimization'', is derived by optimization of the gradient's norm as a surrogate function for \eqref{eq: finite sum}, i.e.,
    \begin{align}
        \label{eq: finite sum grad}
        \min_{\ww \in \reals^{d}} 
        \Bigg\{ \frac{1}{2} \big\|\grad f(\ww)\big\|^2 
        = \frac{1}{2 m^{2}} \norm{\sum_{i = 1}^{m} \grad f_{i}(\ww)}^{2} \Bigg\}.
    \end{align}
    
    When $f$ is \textit{invex}, \cite{ben1986invexity}, the problems \cref{eq: finite sum} and \cref{eq: finite sum grad} have the same solutions. Recall that invexity is the generalization of convexity, which extends the sufficiency of the first
    order optimality condition, e.g., Karush-Kuhn-Tucker conditions, to a broader class of
    problems than simple convex programming. In other words, invexity is a special case of non-convexity, which subsumes convexity as a sub-class; see \cite{mishra2008invexity} for a detailed treatment of invexity.
    In this light, unlike DiSCO and GIANT, by considering the surrogate function \eqref{eq: finite sum grad}, DINGO's application range and theoretical guarantees extend far beyond convex settings to invex problems. 
    Naturally, by considering \eqref{eq: finite sum grad}, DINGO may converge to a local maximum or saddle point in non-invex problems.
    
    As with GIANT and DiSCO, and unlike DANE, InexactDANE and AIDE, our algorithm involves a few hyper-parameters that are easy to tune. Moreover, we theoretically show that DINGO is not too sensitive to the choice of its hyper-parameters in that a strict reduction in the gradient norm is guaranteed, regardless of the selected hyper-parameters.
    Further, the underlying sub-problems of DINGO are simple linear least-squares, for which a plethora of efficient algorithms exist.
    However, unlike GIANT and DiSCO, DINGO does not impose any specific form on the underlying functions, and its theoretical analysis and application range extend beyond convexity. Also, unlike GIANT, the distribution of the data across the computing environment can be arbitrary, i.e., we allow for $|S_{i}| < d$. 
    See Tables~\ref{table: 1} and \ref{table: 2} for a summary of high-level algorithm properties. 
    
    \begin{table}[t]
    \caption{Comparison of problem class, function form and data distribution. 
    Note that DINGO doesn't assume invexity in analysis, rather it is suited to invex problems in practice. 
    }\label{table: 1}
    \vspace{10pt}
    \centering
    \begin{tabular}{@{} lccc @{}}
        \toprule
        & Problem Class & Form of $ \ell_j(\ww; \xx_{j}) $ & Data Distribution \\
        \midrule
        \textbf{DINGO} & Invex & Any & Any \\
        \textbf{GIANT} \cite{GIANT} & Strongly-convex & $\psi_j(\langle \ww, \xx_{j} \rangle) + \gamma \|\ww\|^{2}$  & $|S_{i}| > d$ \\
        \textbf{DiSCO} \cite{DiSCO} & Strongly-convex & Any & Any \\
        \textbf{InexactDANE} \cite{AIDE} & General Non-convex & Any & Any \\
        \textbf{AIDE} \cite{AIDE} & General Non-convex & Any & Any \\
        \bottomrule
    \end{tabular}
\end{table}
    \begin{table}[t]
    \caption{Comparison of number of sub-problem hyper-parameters and communication rounds per iteration. 
    Under inexact update, the choice of sub-problem solver will determine additional hyper-parameters.
    Most communication rounds of DiSCO arise when iteratively solving its sub-problem.
    We assume DINGO and GIANT use two communication rounds for line-search per iteration.
    }\label{table: 2}
    \vspace{10pt}
    \centering
    \begin{tabular}{@{} lcc @{}}
        \toprule
        & \multirow{3}{5cm}{\centering Number of Sub-problem Hyper-Parameters \\ (Under Exact Update)} & \multirow{3}{6cm}{\centering Communication Rounds \\ Per Iteration \\ (Under Inexact Update)} \\ \\ \\
        \midrule
        \textbf{DINGO} & $2$ & $\leq 8$ \\
        \textbf{GIANT} \cite{GIANT} & $0$ & $6$ \\
        \textbf{DiSCO} \cite{DiSCO} & $0$ & $2+2 \times (\text{sub-problem iterations})$ \\
        \textbf{InexactDANE} \cite{AIDE} & $2$ & $4$ \\
        \textbf{AIDE} \cite{AIDE} & $3$ & $4 \times (\text{inner InexactDANE iterations})$ \\
        \bottomrule
    \end{tabular}
\end{table}

    \subsubsection*{Distributed Environment}
    The distributed computing environment that we consider is also assumed by GIANT, DiSCO, DANE, InexactDANE and AIDE.
    Moreover, as with these methods, we restrict communication to vectors of size linear in $d$, i.e., $\mathcal{O}(d)$.
    A \textit{communication round} is performed when the driver uses a \textit{broadcast} operation to send information to one or more workers in parallel, or uses a \textit{reduce} operation to receive information from one or more workers in parallel.
    For example, computing the gradient at iteration $t$, namely $\bgg_t = \sum_{i=1}^{m}\bgg_{t,i}/m$, requires two communication rounds, i.e., the driver broadcasts $\ww_t$ to all workers and then, by a reduce operation, receives $\bgg_{t,i}$ for all $i$.
    We further remind that the distributed computational model considered here is such that \emph{the main bottleneck involves the communications across the network}.

\section{DINGO}\label{Section: Our Algorithm}
    In this section, we describe the derivation of DINGO. 
    This method is depicted in Algorithm~\ref{alg: Our Method} for exact update, while inexact update conditions can be seamlessly incorporated following the discussion in Section~\ref{section: update direction}.
    The assumptions made on \eqref{eq: finite sum} along with the theoretical analysis are discussed in \cref{Section: Theoretical Analysis}.
    Each iteration $t$ involves the computation of two main ingredients: an \textit{update direction} $\pp_t$, and an appropriate \textit{step-size} $\alpha_t$. As usual, our next iterate is then set as $\ww_{t+1}=\ww_t+\alpha_t\pp_t$.
    Below, we provide details for computations of $\pp_t$ and $\alpha_t$.
    
    \subsection{Update Direction\texorpdfstring{: $\pp_t$}{}}\label{section: update direction}
        We begin iteration $t$ by distributively computing the gradient $\bgg_t = \sum_{i=1}^{m}\bgg_{t,i}/m$.
        Thereafter, the driver broadcasts $\bgg_t$ to all workers and in parallel each worker $i$ computes
        $\HH_{t,i}\bgg_t$, $\HH_{t,i}^\dagger\bgg_t$ and $\tilde{\HH}_{t,i}^\dagger\tilde{\bgg}_t$. 
        By a reduce operation, the driver computes the Hessian-gradient product
        $\HH_t\bgg_t = \sum_{i=1}^{m}\HH_{t,i}\bgg_t / m$ as well as the vectors
        $\sum_{i=1}^{m}\HH_{t,i}^\dagger\bgg_t / m$ and 
        $\sum_{i=1}^{m}\tilde{\HH}_{t,i}^\dagger\tilde{\bgg}_t / m$. 
        Computing the update direction $\pp_t$ involves three cases, all of which involve simple linear least-squares sub-problems: 
        
        \begin{enumerate}[label = \textbf{Case \arabic*} ]
        \item \label{Case 1}
            If 
            \begin{equation*}
                \bigg\langle \frac{1}{m}\sum_{i=1}^{m}\HH_{t,i}^\dagger\bgg_t, \HH_t\bgg_t \bigg\rangle \geq \theta\|\bgg_t\|^2,
            \end{equation*}
            then we let $\pp_t = \sum_{i=1}^{m}\pp_{t,i} / m$, with $\pp_{t,i}=-\HH_{t,i}^{\dagger}\bgg_t$.
            Here, we check that the potential update direction ``$-\sum_{i=1}^{m}\HH_{t,i}^{\dagger}\bgg_t/m$" is a suitable descent direction for our surrogate objective \eqref{eq: finite sum grad}. 
            We do this since we have not imposed any restrictive assumptions on \eqref{eq: finite sum} that would automatically guarantee descent; see Lemma~\ref{Lemma: always case 1} for an example of such restrictive assumptions.

        \item \label{Case 2}
            If \labelcref{Case 1} fails, we include regularization and check again that the new potential update direction yields suitable descent.
            Namely, if
            \begin{equation*}
                \bigg\langle \frac{1}{m}\sum_{i=1}^{m}\tilde{\HH}_{t,i}^\dagger\tilde{\bgg}_t, \HH_t\bgg_t \bigg\rangle \geq \theta\|\bgg_t\|^2,
            \end{equation*}
            then we let $\pp_t = \sum_{i=1}^{m}\pp_{t,i} / m$, with $\pp_{t,i}=-\tilde{\HH}_{t,i}^{\dagger}\tilde{\bgg}_t$.

        \item \label{Case 3}
            If all else fails, we enforce descent in the norm of the gradient. 
            Specifically, as \labelcref{Case 2} does not hold, the set
            \begin{equation}\label{eq: set of Case 3 iteration indices}
                \mathcal{I}_t 
                \defeq 
                \big\{ i=1,2,\ldots,m \mid \langle \tilde{\HH}_{t,i}^\dagger\tilde{\bgg}_t, \HH_t\bgg_t \rangle < \theta\|\bgg_t\|^2 \big\},
            \end{equation}
            is non-empty.
            In parallel, the driver broadcasts $\HH_t\bgg_t$ to each worker $i\in\mathcal{I}_t$ and has it locally compute the solution to
            \begin{equation}\label{eq: Lagrange Sub Problem 1}
            	\argmin_{\pp_{t,i}} \frac{1}{2} \|\HH_{t,i}\pp_{t,i}+\bgg_t\|^2 + \frac{\phi^2}{2}\|\pp_{t,i}\|^2, 
            	\quad\textrm{such that}\quad 
                \langle \pp_{t,i},\HH_t\bgg_t \rangle \leq -\theta\|\bgg_t\|^2,
            \end{equation}
            where $\phi$ is as in Algorithm~\ref{alg: Our Method}.
            We now solve the sub-problem \eqref{eq: Lagrange Sub Problem 1} using the method of Lagrange multipliers \cite{nocedal2006numerical}.
            We rewrite \eqref{eq: Lagrange Sub Problem 1} as
            \begin{equation*}
                \argmin_{\pp_{t,i}} \frac{1}{2} \|\tilde{\HH}_{t,i}\pp_{t,i}+\tilde{\bgg}_t\|^2, \quad\textrm{such that}\quad 
                \langle \pp_{t,i},\HH_t\bgg_t \rangle \leq -\theta\|\bgg_t\|^2,
            \end{equation*}
            where $\tilde{\HH}_{t,i}$ and $\tilde{\bgg}_t$ are defined in \eqref{eq: H tilde and g tilde}. 
            Let
            \begin{equation}\label{eq: Lagrange Sub Problem 2}
                \mathcal{L}(\pp_{t,i},\lambda_{t,i}) 
                = \frac{1}{2} \|\tilde{\HH}_{t,i}\pp_{t,i}+\tilde{\bgg}_t\|^2 
                + \lambda_{t,i}\big( \langle\pp_{t,i},\HH_t\bgg_t\rangle + \theta\|\bgg_t\|^2 \big).
            \end{equation}
            Therefore,  
            $\zero 
            = \nabla_{\pp_{t,i}} \mathcal{L}(\pp_{t,i},\lambda_{t,i}) 
            = \tilde{\HH}_{t,i}^T(\tilde{\HH}_{t,i}\pp_{t,i}+\tilde{\bgg}_t) + \lambda_{t,i} \HH_t\bgg_t$ 
            if and only if
            \begin{equation}\label{eq: Lagrange Sub Problem 3}
                \tilde{\HH}_{t,i}^T\tilde{\HH}_{t,i}\pp_{t,i} = -\tilde{\HH}_{t,i}^T\tilde{\bgg}_t - \lambda_{t,i} \HH_t\bgg_t.
            \end{equation}
            Due to regularization by $\phi$, the matrix $\tilde{\HH}_{t,i}^T$ has full column rank 
            and $\tilde{\HH}_{t,i}^T\tilde{\HH}_{t,i}$ is invertible.
            Therefore, the unique solution to \eqref{eq: Lagrange Sub Problem 3} is
            \begin{equation}\label{eq: Lagrange Sub Problem 4}
                \pp_{t,i} 
                = -(\tilde{\HH}_{t,i}^T\tilde{\HH}_{t,i})^{-1}\tilde{\HH}_{t,i}^T\tilde{\bgg}_t 
                    - \lambda_{t,i} (\tilde{\HH}_{t,i}^T\tilde{\HH}_{t,i})^{-1} \HH_t\bgg_t
                = -\tilde{\HH}_{t,i}^\dagger\tilde{\bgg}_t - \lambda_{t,i}(\tilde{\HH}_{t,i}^T\tilde{\HH}_{t,i})^{-1} \HH_t\bgg_t.
            \end{equation}

            Assumption~\ref{assumption: GHNSP} implies that for $ \bgg_t \neq \zero $, we have $ \HH_t\bgg_t \neq \zero $, which in turn gives 
            \begin{equation*}
                \big\langle (\tilde{\HH}_{t,i}^T\tilde{\HH}_{t,i})^{-1}\HH_t\bgg_t,\HH_t\bgg_t \big\rangle > 0.
            \end{equation*} 
            Thus, 
            $0 
            = 
            \frac{\partial}{\partial \lambda_{t,i}} \mathcal{L}(\pp_{t,i},\lambda_{t,i})
            = 
            - \lambda_{t,i} \big\langle (\tilde{\HH}_{t,i}^T\tilde{\HH}_{t,i})^{-1}\HH_t\bgg_t,\HH_t\bgg_t \big\rangle 
            - \langle \tilde{\HH}_{t,i}^\dagger\tilde{\bgg}_t, \HH_t\bgg_t \rangle
            + \theta\|\bgg_t\|^2$
            if and only if
            \begin{equation*}
                \lambda_{t,i}
                = 
                \frac{- \langle \tilde{\HH}_{t,i}^\dagger\tilde{\bgg}_t, \HH_t\bgg_t \rangle + \theta\|\bgg_t\|^2}
                    {\big\langle (\tilde{\HH}_{t,i}^T\tilde{\HH}_{t,i})^{-1}\HH_t\bgg_t,\HH_t\bgg_t \big\rangle}.
            \end{equation*}
            Therefore, the solution to the sub-problem \eqref{eq: Lagrange Sub Problem 1} is
            \begin{equation}\label{eq: case 3 update direction}
                \pp_{t,i} 
                = 
                -\tilde{\HH}_{t,i}^\dagger\tilde{\bgg}_t - \lambda_{t,i}(\tilde{\HH}_{t,i}^T\tilde{\HH}_{t,i})^{-1} \HH_t\bgg_t, 
                \quad\text{where}\quad
                \lambda_{t,i} 
                = 
                \frac{- \langle \tilde{\HH}_{t,i}^\dagger\tilde{\bgg}_t, \HH_t\bgg_t \rangle + \theta\|\bgg_t\|^2}
                    {\big\langle (\tilde{\HH}_{t,i}^T\tilde{\HH}_{t,i})^{-1}\HH_t\bgg_t,\HH_t\bgg_t \big\rangle} > 0.
            \end{equation}
            The term $\lambda_{t,i}$ in \eqref{eq: case 3 update direction} is positive by the definition of $\mathcal{I}_t$.
            In conclusion, for \labelcref{Case 3}, each worker $i\in\mathcal{I}_t$ computes \eqref{eq: case 3 update direction} and, using a reduce operation, the driver then computes the update direction $\pp_t = \sum_{i=1}^{m}\pp_{t,i} / m$, which by construction yields descent in the surrogate objective \eqref{eq: finite sum grad}. 
            Note that $\pp_{t,i} = -\tilde{\HH}_{t,i}^\dagger\tilde{\bgg}_t$ for all $i\notin\mathcal{I}_t$ have already been obtained as part of \labelcref{Case 2}.
            
            \begin{remark}
                The three cases help avoid the need for any unnecessary assumptions on data distribution or the knowledge of any practically unknowable constants. 
                In fact, given Lemma~\ref{Lemma: always case 1}, which imposes a certain assumption on the data distribution, we could have stated our algorithm in its simplest form, i.e., with only \labelcref{Case 1}. 
                This would be more in line with some prior works, e.g., GIANT, but it would have naturally restricted the applicability of our method in terms of data distributions.
            \end{remark}
    \end{enumerate}
    
    \begin{algorithm}
	\caption{DINGO}\label{alg: Our Method}
    \begin{algorithmic}[1] 
    	\State{\textbf{input} initial point $\ww_0\in\mathbb{R}^d$, 
        	gradient tolerance $\delta\geq0$, maximum iterations $T$, 
        	line search parameter $\rho\in(0,1)$, parameter $\theta>0$ 
        	and regularization parameter $\phi>0$ as in \eqref{eq: H tilde and g tilde}.}
    	\For{$t=0,1,2,\ldots,T-1$}
            \State{Distributively compute the full gradient $\bgg_t = \frac{1}{m}\sum_{i=1}^{m}\bgg_{t,i}$.}
            \If{$\|\bgg_t\| \leq \delta$}
                \State{\textbf{return} $\ww_t$}
            \Else
                \State{The driver broadcasts $\bgg_t$ and, in parallel, each worker $i$ computes $\HH_{t,i}\bgg_t$, $\HH_{t,i}^\dagger\bgg_t$ 
                \hspace*{35pt}and $\tilde{\HH}_{t,i}^{\dagger}\tilde{\bgg}_t$.}
                \State{By a reduce operation, the driver computes 
                $\HH_t\bgg_t = \frac{1}{m}\sum_{i=1}^{m}\HH_{t,i}\bgg_t$, 
                $\frac{1}{m}\sum_{i=1}^{m}\HH_{t,i}^\dagger\bgg_t$ 
                \hspace*{35pt}and $\frac{1}{m}\sum_{i=1}^{m}\tilde{\HH}_{t,i}^\dagger\tilde{\bgg}_t$.}
                \If{$\big\langle \frac{1}{m}\sum_{i=1}^{m}\HH_{t,i}^{\dagger}\bgg_t, \HH_t\bgg_t \big\rangle \geq \theta\|\bgg_t\|^2$}
                    \State{Let $\pp_t = \frac{1}{m}\sum_{i=1}^{m}\pp_{t,i}$, with $\pp_{t,i}=-\HH_{t,i}^{\dagger}\bgg_t$.}
                \ElsIf{$\big\langle \frac{1}{m}\sum_{i=1}^{m}\tilde{\HH}_{t,i}^{\dagger}\tilde{\bgg}_t, \HH_t\bgg_t \big\rangle \geq \theta\|\bgg_t\|^2$}
                    \State{Let $\pp_t = \frac{1}{m}\sum_{i=1}^{m}\pp_{t,i}$, with $\pp_{t,i}=-\tilde{\HH}_{t,i}^{\dagger}\tilde{\bgg}_t$.}
                \Else
                    \State{The driver computes $\pp_{t,i} = -\tilde{\HH}_{t,i}^\dagger\tilde{\bgg}_t$ for all 
                        $i$ such that $\langle \tilde{\HH}_{t,i}^{\dagger}\tilde{\bgg}_t, \HH_t\bgg_t \rangle \geq \theta\|\bgg_t\|^2$.}
                    \State{The driver broadcasts $\HH_t\bgg_t$ to each worker $i$ such that 
                        $\langle \tilde{\HH}_{t,i}^{\dagger}\tilde{\bgg}_t, \HH_t\bgg_t \rangle < \theta\|\bgg_t\|^2$ 
                        \hspace*{53pt}and, in parallel, they compute
                        \begin{equation*}
                            \hspace*{53pt}\pp_{t,i} 
                            = -\tilde{\HH}_{t,i}^\dagger\tilde{\bgg}_t 
                                - \lambda_{t,i}(\tilde{\HH}_{t,i}^T\tilde{\HH}_{t,i})^{-1}\HH_t\bgg_t, 
                            \quad
                            \lambda_{t,i}
                            = \frac{- \langle \tilde{\HH}_{t,i}^\dagger\tilde{\bgg}_t, \HH_t\bgg_t \rangle + \theta\|\bgg_t\|^2}
                            {\big\langle (\tilde{\HH}_{t,i}^T\tilde{\HH}_{t,i})^{-1}\HH_t\bgg_t,\HH_t\bgg_t \big\rangle}.
                        \end{equation*}
                    }
                    \State{Using a reduce operation, the driver computes $\pp_t = \frac{1}{m}\sum_{i=1}^{m}\pp_{t,i}$.}
                \EndIf
                \State{Choose the largest $\alpha_t\in(0,1]$ such that 
                    \begin{equation*}
                        \hspace*{35pt}\big\|\nabla f(\ww_{t}+\alpha_t\pp_t)\big\|^2 
                        \leq 
                        \|\bgg_t\|^2 + 2\alpha_t\rho\langle \pp_t,\HH_t\bgg_t \rangle.
                    \end{equation*}}
                \State{The driver computes $\ww_{t+1} = \ww_{t} + \alpha_t\pp_t$.}
            \EndIf
        \EndFor
        \State{\textbf{return} $\ww_{T}$.}
    \end{algorithmic}
\end{algorithm}

    \subsubsection*{Inexact Update}
    In practice, it is unreasonable to assume that the sub-problems of DINGO, namely $\HH_{t,i}^\dagger\bgg_t$, $\tilde{\HH}_{t,i}^\dagger\tilde{\bgg}_t$ and $(\tilde{\HH}_{t,i}^T\tilde{\HH}_{t,i})^{-1}(\HH_t\bgg_t)$, can be computed exactly for all iterations $t$ and on all workers $i=1,\ldots,m$.
    For example, our implementation of DINGO, which is discussed in detail in Section~\ref{section: Implementation Details}, uses efficient iterative least-squares solvers to approximately solve the sub-problems.
    Using these solvers, we only require Hessian-vector products, instead of forming the Hessian explicitly, which makes DINGO suitable for large dimension $d$; however, this means that inexact solutions will be computed in most applications.
    When approximate solutions to the sub-problems are used, instead of exact solutions, we say that we are under \emph{inexact update}. 
    The level of precision we require by the approximate solutions is presented in Condition~\ref{condition: inexactness condition}.
    
    \begin{condition}[Inexactness Condition]\label{condition: inexactness condition}
        For all iterations $t$, all worker machines $i=1,\ldots,m$ are able to compute approximations $\vv_{t,i}^{(1)}$, $\vv_{t,i}^{(2)}$ and $\vv_{t,i}^{(3)}$ to $\HH_{t,i}^\dagger\bgg_t$, $\tilde{\HH}_{t,i}^\dagger\tilde{\bgg}_t$ and $(\tilde{\HH}_{t,i}^T\tilde{\HH}_{t,i})^{-1}(\HH_t\bgg_t)$, respectively, that satisfy:
    \begin{subequations}\label{eqs: inexactness conditions}
        \begin{align}
            \|\HH_{t,i}^2\vv_{t,i}^{(1)} - \HH_{t,i}\bgg_t\|
            &\leq \varepsilon_i^{(1)}\|\HH_{t,i}\bgg_t\|, 
            \; \vv_{t,i}^{(1)}\in\mathcal{R}(\HH_{t,i}), \label{eq: Case 1 inexactness condition} \\
            \|\tilde{\HH}_{t,i}^T\tilde{\HH}_{t,i}\vv_{t,i}^{(2)} - \HH_{t,i}\bgg_t\|
            &\leq \varepsilon_i^{(2)}\|\HH_{t,i}\bgg_t\|, \label{eq: Case 2 inexactness condition} \\
            \|\tilde{\HH}_{t,i}^T\tilde{\HH}_{t,i}\vv_{t,i}^{(3)} - \HH_t\bgg_t\|
            &\leq \varepsilon_i^{(3)}\|\HH_t\bgg_t\|, 
            \; \langle \vv_{t,i}^{(3)}, \HH_t\bgg_t \rangle > 0, \label{eq: Case 3 inexactness condition}
        \end{align}
    \end{subequations}
    where $0\leq\varepsilon_i^{(1)},\varepsilon_i^{(2)},\varepsilon_i^{(3)}<1$ are constants.
    \end{condition}
    
    The inexactness conditions \eqref{eqs: inexactness conditions} are selected, mainly, for practical purposes as they are achievable and verifiable in real-world use.
    Using more direct requirements such as $\|\vv_{t,i}^{(1)} - \HH_{t,i}^\dagger\bgg_t\| \leq \varepsilon_i^{(1)}\|\HH_{t,i}^\dagger\bgg_t\|$ instead of \eqref{eq: Case 1 inexactness condition}, while being more straightforward to use in convergence theory, is not verifiable in practice as it requires knowing the vector we are trying to compute.
    The conditions \eqref{eqs: inexactness conditions} also maintain a desired theoretical property in that the exact solutions, namely $\vv_{t,i}^{(1)}=\HH_{t,i}^\dagger\bgg_t$, $\vv_{t,i}^{(2)}=\tilde{\HH}_{t,i}^\dagger\tilde{\bgg}_t$ and $\vv_{t,i}^{(3)}=(\tilde{\HH}_{t,i}^T\tilde{\HH}_{t,i})^{-1}(\HH_t\bgg_t)$, satisfy \eqref{eqs: inexactness conditions} with $\varepsilon_i^{(1)}=\varepsilon_i^{(2)}=\varepsilon_i^{(3)}=0$.
    By having $\varepsilon_i^{(1)},\varepsilon_i^{(2)},\varepsilon_i^{(3)}<1$, we are ensuring that we ``do better" than the approximations $\vv_{t,i}^{(1)} = \vv_{t,i}^{(2)} = \vv_{t,i}^{(3)} = \zero$. 
    Under Condition~\ref{condition: inexactness condition}, DINGO will achieve a strict reduction in the gradient norm in practice, regardless of the selected hyper-parameters.
    However, to guarantee an overall worst-case \emph{linear convergence rate} of DINGO under inexact update in Corollary~\ref{corollary: unify inexact}, we place an upper bound on $\varepsilon_i^{(3)}$, which can be increased arbitrarily close to $1$ by increasing $\phi$, while $\varepsilon_i^{(1)}$ and $\varepsilon_i^{(2)}$ are allowed to be arbitrary.
    
    Running DINGO under inexact update is analogous to running it under exact update and is formulated from the same intentions.
    Namely, we begin iteration $t$ by distributively computing the gradient $\bgg_t$. 
    Thereafter, we distributively compute the Hessian-gradient product $\HH_t\bgg_t$ as well as the vectors $\sum_{i=1}^{m}\vv_{t,i}^{(1)}/m$ and $\sum_{i=1}^{m}\vv_{t,i}^{(2)}/m$. 
    Computing the update direction $\pp_t$ involves three cases:
    
    \begin{enumerate}[label = \textbf{Case \arabic**} ]
        \item \label{Case 1*}
            If 
            \begin{equation*}
                \bigg\langle \frac{1}{m}\sum_{i=1}^{m}\vv_{t,i}^{(1)}, \HH_t\bgg_t \bigg\rangle 
                \geq \theta\|\bgg_t\|^2,
            \end{equation*}
            then we let $\pp_t = \sum_{i=1}^{m}\pp_{t,i} / m$, with $\pp_{t,i}=-\vv_{t,i}^{(1)}$.
        
        \item \label{Case 2*}
            If \labelcref{Case 1*} fails, we include regularization and check again that the new potential update direction yields suitable descent.
            Namely, if
            \begin{equation*}
                \bigg\langle \frac{1}{m}\sum_{i=1}^{m}\vv_{t,i}^{(2)}, \HH_t\bgg_t \bigg\rangle 
                \geq \theta\|\bgg_t\|^2,
            \end{equation*}
            then we let $\pp_t = \sum_{i=1}^{m}\pp_{t,i} / m$, with $\pp_{t,i}=-\vv_{t,i}^{(2)}$.
            
        \item \label{Case 3*}
            If all else fails, we enforce descent in the norm of the gradient. 
            More specifically, the driver broadcasts $\HH_t\bgg_t$ to all workers $i$ in 
            \begin{equation}\label{eq: set of Case 3* iteration indices}
                \mathcal{I}_{t}^* 
                \defeq 
                \big\{ i=1,\ldots,m \mid \langle \vv_{t,i}^{(2)}, \HH_t\bgg_t \rangle < \theta\|\bgg_t\|^2 \big\},
            \end{equation}
            and they locally compute
            \begin{equation}\label{eq: case 3* update direction}
                \pp_{t,i} 
                = - \vv_{t,i}^{(2)} - \lambda_{t,i}\vv_{t,i}^{(3)}, 
                \quad\text{where}\quad
                \lambda_{t,i} 
                = \frac{-\langle \vv_{t,i}^{(2)}, \HH_t\bgg_t \rangle + \theta\|\bgg_t\|^2}
                {\langle \vv_{t,i}^{(3)}, \HH_t\bgg_t \rangle} 
                > 0.
            \end{equation}
            The term $\lambda_{t,i}$ in \eqref{eq: case 3* update direction} is positive by the definition of $\mathcal{I}_{t}^*$ and the condition in \eqref{eq: Case 3 inexactness condition}.
            Assumption~\ref{assumption: GHNSP} implies that for $ \bgg_t \neq \zero $, we have $ \HH_t\bgg_t \neq \zero $.
            By a reduce operation, the driver then computes the update direction $\pp_t = \sum_{i=1}^{m}\pp_{t,i} / m$, which by construction yields descent in the surrogate objective \eqref{eq: finite sum grad}. 
            Note that $\pp_{t,i} = -\vv_{t,i}^{(2)}$ for all $i\notin\mathcal{I}_t^*$ have already been obtained as part of \labelcref{Case 2*}.
    \end{enumerate}
    
    \subsection{Step Size\texorpdfstring{: $\alpha_t$}{}}\label{section: line-search}
        After computing the update direction $\pp_t$, DINGO computes the next iterate, $ \ww_{t+1} $, by moving along $ \pp_{t,i} $ by an appropriate step-size $\alpha_t$ and forming $\ww_{t+1} = \ww_t+\alpha_t\pp_t$. 
        We use an Armijo-type line-search to choose this step-size. 
        Specifically, as we are minimizing the norm of the gradient as a surrogate function, we choose the largest $\alpha_t\in(0,1]$ such that
        \begin{equation}\label{eq: Armijo-type line seach}
            \|\bgg_{t+1}\|^2 \leq \|\bgg_t\|^2 + 2\alpha_t\rho \langle \pp_t,\HH_t\bgg_t \rangle,
        \end{equation}
        for some constant $\rho\in(0,1)$. 
        By construction of $\pp_t$ we always have $\langle \pp_t,\HH_t\bgg_t \rangle \leq -\theta\|\bgg_t\|^2$, which implies
        \begin{equation*}
            \|\bgg_{t+1}\|^2 \leq (1-2\alpha_t\rho\theta) \|\bgg_t\|^2 < \|\bgg_t\|^2.
        \end{equation*}
        Therefore, after each iteration we are strictly decreasing the norm of the gradient of \eqref{eq: finite sum}, and line-search guarantees that this occurs irrespective of all hyper-parameters of \cref{alg: Our Method}, i.e., $ \theta $ and $ \phi $.

\section{Theoretical Analysis}\label{Section: Theoretical Analysis}
    In this section, we derive convergence results for DINGO, under exact update and under inexact update.
    Unless otherwise stated, all proofs can be found in Appendix~\ref{appendix: proofs}.
    Recall, from Section~\ref{section: update direction}, that three potential cases may occur when computing the update direction~$\pp_t$.
    The convergence analysis under these cases are treated separately in Sections~\ref{section: case 1}, \ref{section: case 2} and \ref{Section: Case 3}.
    We begin, in Section~\ref{section: general assumptions}, by establishing a general underlying assumption for our analysis. 
    The analysis of \labelcref{Case 1} and \labelcref{Case 3}, and their inexact variants, require their own specific assumptions, which are discussed in Sections~\ref{section: case 1} and \ref{Section: Case 3}, respectively.

    \subsection{General Assumption}\label{section: general assumptions}
        In this section, we discuss a general assumption on \eqref{eq: finite sum}, which underlies our analysis of DINGO. 
        We make the following standard assumption on the local gradients and local Hessian matrices.
        
        \begin{assumption}[Local Lipschitz Continuity of Gradient and Hessian]\label{assumption: Lipschitz Continuity}
            The function $f_i$ in \eqref{eq: finite sum} is twice differentiable for all $i=1,\ldots,m$.
            Moreover, for all $i=1,\ldots,m$, there exists constants $K_i,L_i\in(0,\infty)$ such that
            \begin{equation*}
                \big\|\grad f_i(\xx)-\grad f_i(\yy)\big\| \leq K_i\|\xx-\yy\| 
                \quad\text{and}\quad
                \big\|\grad^2f_i(\xx)-\grad^2f_i(\yy)\big\| \leq L_i\|\xx-\yy\|,
            \end{equation*}
            for all $\xx,\yy\in\mathbb{R}^d$.
        \end{assumption}
        
        Assumption~\ref{assumption: Lipschitz Continuity} implies
        \begin{align*}
            \big\|\grad^2f(\yy)\grad f(\yy) - \grad^2f(\xx)\grad f(\xx)\big\|
            &\leq \big\|\grad^2f(\yy)\big\| \big\|\grad f(\xx)-\grad f(\yy)\big\| \\
            &\quad + \big\|\grad f(\xx)\big\| \big\|\grad^2f(\xx)-\grad^2f(\yy)\big\| \\
            &\leq \Bigg(
                \bigg(\frac{1}{m}\sum_{i=1}^{m}K_i\bigg)^2 
                + \bigg(\frac{1}{m}\sum_{i=1}^{m}L_i\bigg) \big\|\grad f(\xx)\big\|
            \Bigg)\|\xx-\yy\|,
        \end{align*}
        for all $\xx,\yy\in\mathbb{R}^d$.
        Therefore, as DINGO achieves a strict reduction in the norm of the gradient, for all iterations $t$ we have
        \begin{equation}\label{eq: Lipschitz to Moral}
            \big\| \grad^2 f(\ww) \grad f(\ww) - \grad^2 f(\ww_t) \grad f(\ww_t) \big\| 
        	\leq  L(\ww_0)\|\ww - \ww_t\|,
        \end{equation}
        for all $\ww\in\bbR^d$, where $L(\ww_0) \defeq \big(\sum_{i=1}^{m}K_i/m\big)^2 + \big(\sum_{i=1}^{m}L_i/m\big) \big\|\grad f(\ww_0)\big\|$ is a constant and $\ww_0$ is the initial point as in Algorithm~\ref{alg: Our Method}.
        In fact, for our results on the convergence analysis of DINGO under exact update (see Theorems~\ref{theorem: Case 1}, \ref{theorem: Case 2} and \ref{theorem: Case 3}), we simply could have assumed that the Hessian-gradient product is Lipschitz continuous on the piecewise-linear path generated by the iterates, instead of Assumption~\ref{assumption: Lipschitz Continuity}.
        
        In our proofs of convergence properties, we will make frequent use of a result from the following well-known lemma; see, for example, \cite{Newton-MR} for a proof.
        
        \begin{lemma}\label{lemma: Moral-Smoothness}
            Let $\xx,\zz\in\bbR^d$, $L\in[0,\infty)$ and $h:\bbR^d\rightarrow\bbR$ be differentiable. 
            If 
            \begin{equation*}
                \big\| \grad h(\yy) - \grad h(\xx) \big\| \leq L \| \yy - \xx \|,
            \end{equation*}
            for all $\yy\in[\xx,\zz]$, then
            \begin{equation*}
                h(\yy) 
                \leq h(\xx) + \big\langle \yy-\xx,\grad h(\xx) \big\rangle 
                + \frac{L}{2}\| \yy-\xx \|^2,
            \end{equation*}
            for all $\yy\in[\xx,\zz]$.
        \end{lemma}
        
        It follows from $\grad \Big(\frac{1}{2}\big\|\grad f(\ww)\big\|^2 \Big) = \grad^2 f(\ww) \grad f(\ww)$, Assumption~\ref{assumption: Lipschitz Continuity}, Lemma~\ref{lemma: Moral-Smoothness} and \eqref{eq: Lipschitz to Moral} that
        \begin{equation}\label{eq: Moral-Smoothness Corollary}
        	\big\|\grad f(\ww_t+\alpha\pp_t)\big\|^2 
        	\leq 
        	\|\bgg_t\|^2 + 2\alpha\big\langle \pp_t,\HH_t\bgg_t \big\rangle + \alpha^2L(\ww_0)\|\pp_t\|^2,
        \end{equation}
        for all $\alpha\geq0$ and all iterations $t$, where $\pp_t$ is the update direction of DINGO at iteration $t$.

    \subsection{Analysis of Case 1 and Case 1*}\label{section: case 1}
        In this section, we analyze the convergence of iterations of DINGO that fall under \hbox{\labelcref{Case 1}} and \hbox{\labelcref{Case 1*}} under the condition of exact and inexact update, respectively. 
        For such iterations, we make the following assumption about the action of the local Hessian matrices $\HH_{t,i}$ on their range space.
        \begin{assumption}[Local Pseudo-Inverse Regularity]\label{assumption: Case 1 Gammas}
        	For all $i=1,\ldots,m$, there exists a constant $\gamma_i\in(0,\infty)$ such that for all $\ww\in\bbR^d$ we have $\big\|\grad^2f_i(\ww)\pp\big\| \geq \gamma_i\big\|\pp\big\|$ for all $\pp\in\mathcal{R}\big(\grad^2f_i(\ww)\big)=\mathcal{N}\big(\grad^2f_i(\ww)\big)^\perp$ with $f_i$ as in \eqref{eq: finite sum}.
        \end{assumption}
        
        It is easy to show that Assumption~\ref{assumption: Case 1 Gammas} is equivalent to $\big\|\grad^2f_i(\ww)^\dagger\big\| \leq \gamma_i^{-1}$.
        Note, in our theory of the convergence rate under \labelcref{Case 1} (see Theorem~\ref{theorem: Case 1}), we could use the weaker assumption that: $\big\|\grad^2f_i(\ww)^\dagger\grad f_i(\ww)\big\| \leq \gamma_i^{-1}\big\|\grad f_i(\ww)\big\|$ for all $\ww\in\bbR^d$ and all $i=1,\ldots,m$.
        However, we require the more general Assumption~\ref{assumption: Case 1 Gammas} in our convergence rate theory under inexact update (see Theorem~\ref{theorem: Case 1 Inexact}).
        
        Recall that Assumption~\ref{assumption: Case 1 Gammas} is a significant relaxation of strong convexity.
        As an example, take an under-determined least squares problem, where we have $f_i(\ww)=\|\AA_i\ww-\bb_i\|^2/2$ in \eqref{eq: finite sum}. 
        In this example, each $f_i$ is clearly not strongly convex; yet, Assumption~\ref{assumption: Case 1 Gammas} is satisfied with $\gamma_i = \sigma_\text{min}^2(\AA_i)$, where $\sigma_\text{min}(\AA_i)$ is the smallest non-zero singular value of $\AA_i$.
        Further discussion on Assumption~\ref{assumption: Case 1 Gammas}, in the case where $m=1$, can be found in \cite{Newton-MR}.

        \subsubsection*{Exact Update}
        Suppose we are able to run DINGO with exact update. 
        For iterations in \labelcref{Case 1}, we obtain the following convergence result.
        
        \begin{theorem}[Convergence Under Case 1]\label{theorem: Case 1}
        	Suppose Assumptions~\ref{assumption: Lipschitz Continuity} and \ref{assumption: Case 1 Gammas} both hold and that we run Algorithm~\ref{alg: Our Method}. 
            Then for all iterations $t$ in \labelcref{Case 1} we have 
            \hbox{$\|\bgg_{t+1}\|^2 \leq (1-2\tau_1\rho\theta) \|\bgg_t\|^2$} with constants
            \begin{equation}\label{eq: Case 1 Convergence}
                \tau_1 = \frac{2(1-\rho)\gamma^2\theta}{L(\ww_0)}, 
            	\quad
            	\gamma = \Bigg(\frac{1}{m}\sum_{i=1}^{m}\frac{1}{\gamma_i}\Bigg)^{-1},
            \end{equation}
            where $L(\ww_0)$ is as in \eqref{eq: Lipschitz to Moral}, $\rho$ and $\theta$ are as in Algorithm~\ref{alg: Our Method}, and $\gamma_i$ are as in Assumption~\ref{assumption: Case 1 Gammas}. 
            Moreover, for iterations to exist in \labelcref{Case 1} it is necessary that $\theta\leq\sqrt{L(\ww_0)}/\gamma$, which implies $0\leq 1-2\tau_1\rho\theta <1$.
        \end{theorem}
        
        \begin{remark}
            In Theorem~\ref{theorem: Case 1}, the term $\gamma$ involves the \emph{average} of the $\gamma_i$'s. 
            This is beneficial as it ``smooths out" non-uniformity in the $\gamma_i$'s; for example, $\gamma\geq\min_i\gamma_i$.
            Moreover, the term $L(\ww_0)$ also involves the \emph{average} of the $K_i$'s and $L_i$'s from Assumption~\ref{assumption: Lipschitz Continuity}.
        \end{remark}
         
        Intuitively, if $\theta$ becomes larger then the condition $\big\langle \frac{1}{m}\sum_{i=1}^{m}\HH_{t,i}^{\dagger}\bgg_t, \HH_t\bgg_t \big\rangle \geq \theta\|\bgg_t\|^2$, of \labelcref{Case 1}, is less likely to hold at any iteration $t$.
        Theorem~\ref{theorem: Case 1} provides a necessary condition on how large $\theta$ can be for iterations to exist in \labelcref{Case 1}.
        Under specific assumptions on \eqref{eq: finite sum}, we can theoretically guarantee that all iterations $t$ are in \hbox{\labelcref{Case 1}}.
        The following lemma provides one such example.
         
        \begin{lemma}
        	\label{Lemma: always case 1}
            Suppose $f_i$ in \eqref{eq: finite sum} is twice differentiable, for all $i=1,\ldots,m$, and that Assumption~\ref{assumption: Case 1 Gammas} holds and we run Algorithm~\ref{alg: Our Method}.
            Furthermore, for all $i=1,\ldots,m$, suppose that the Hessian matrix $\HH_{t,i}$ is invertible and there exists a constant $\varepsilon_i\in[0,\infty)$ such that $\|\HH_{t,i}-\HH_t\|\leq\varepsilon_i$, for all iterations $t$.
            If
            \begin{equation*}
                \frac{1}{m}\sum_{i=1}^{m}\bigg(1-\frac{\varepsilon_i}{\gamma_i}\bigg) \geq \theta,
            \end{equation*}
            where $\gamma_i$ are as in Assumption~\ref{assumption: Case 1 Gammas}, then all iterations $t$ are in \labelcref{Case 1}.
        \end{lemma}
        Under the assumptions of \cref{Lemma: always case 1}, if the Hessian matrix for each worker is \emph{on average} a reasonable approximation to the full Hessian, i.e., $\varepsilon_i$ is on average sufficiently small so that 
        \begin{align*}
        \frac{1}{m}\sum_{i=1}^{m}\frac{\varepsilon_i}{\gamma_i} < 1,
        \end{align*}
        then we can choose $\theta$ small enough to ensure that all iterations $t$ are in \labelcref{Case 1}. 
        In other words, for the iterates to stay in \labelcref{Case 1}, we do not require the Hessian matrix of each individual worker to be a high-quality approximation to the full Hessian (which could indeed be hard to enforce in many practical applications). 
        As long as the data is distributed in such a way that Hessian matrices are on average reasonable approximations, we can guarantee to have all iterations in \labelcref{Case 1}.

        \subsubsection*{Inexact Update}
        We now suppose that we run DINGO with inexact update.
        For iterations in \labelcref{Case 1*}, recall that this is the inexact variant of \labelcref{Case 1}, we obtain the following convergence result.
        
        \begin{theorem}[Convergence Under Case 1*]\label{theorem: Case 1 Inexact}
            Suppose Assumptions~\ref{assumption: Lipschitz Continuity} and \ref{assumption: Case 1 Gammas} hold and that we run Algorithm~\ref{alg: Our Method} with inexact update such that the inexactness condition in \eqref{eq: Case 1 inexactness condition} holds.
            Then for all iterations $t$ in \labelcref{Case 1*} we have $\|\bgg_{t+1}\|^2 \leq (1-2\tilde{\tau}_1\rho\theta) \|\bgg_t\|^2$ with constants
            \begin{equation}\label{eq: Case 1 Inexact Convergence}
            	\tilde{\tau}_1 = \frac{2(1-\rho)\gamma^2\theta}{L(\ww_0)},
            	\quad
            	\gamma = \Bigg(\frac{1}{m}\sum_{i=1}^{m}\frac{\gamma_i+\varepsilon_i^{(1)}K_i}{\gamma_i^2}\Bigg)^{-1},
            \end{equation}
            where $L(\ww_0)$ is as in \eqref{eq: Lipschitz to Moral}, $\rho$ and $\theta$ are as in Algorithm~\ref{alg: Our Method}, $K_i$ are as in Assumption~\ref{assumption: Lipschitz Continuity}, $\gamma_i$ are as in Assumption~\ref{assumption: Case 1 Gammas}, and $\varepsilon_i^{(1)}$ are as in \eqref{eq: Case 1 inexactness condition}.
            Moreover, for iterations to exist in \labelcref{Case 1*} it is necessary that $\theta\leq\sqrt{L(\ww_0)}/\gamma$, which implies $0\leq 1-2\tilde{\tau}_1\rho\theta <1$.
        \end{theorem}
        
        Theorem~\ref{theorem: Case 1} gives a faster convergence rate than Theorem~\ref{theorem: Case 1 Inexact} as $\tau_1\geq\tilde{\tau}_1$, where $\tau_1$ and $\tilde{\tau}_1$ are as in Theorem~\ref{theorem: Case 1} and Theorem~\ref{theorem: Case 1 Inexact}, respectively.
        Reducing the inexactness error $\varepsilon_i^{(1)}$ leads to faster convergence in Theorem~\ref{theorem: Case 1 Inexact}.
        Moreover, if $\HH_{t,i}^\dagger\bgg_t$ is computed exactly, that is $\varepsilon_i^{(1)}=0$ for all $i=1,\ldots,m$, then Theorem~\ref{theorem: Case 1 Inexact} gives the same convergence as Theorem~\ref{theorem: Case 1}.

    \subsection{Analysis of Case 2 and Case 2*}\label{section: case 2}
        We now analyze the convergence of DINGO for iterations that fall under \labelcref{Case 2} and \hbox{\labelcref{Case 2*}} under the condition of exact and inexact update, respectively.
        For this analysis, we do not require any additional assumptions to that of Assumption~\ref{assumption: Lipschitz Continuity}. 
        Instead, Lemma~\ref{lem: pseudoinverse norm bound}, in Section~\ref{section: proof of case 2}, implies the upper bound
        \begin{equation}\label{eq: H tilde dagger norm bound}
            \|\tilde{\HH}_{t,i}^\dagger\| \leq \frac{1}{\phi},
        \end{equation}
        for all iterations $t$ and all $i=1,\ldots,m$, where $\phi$ is as in Algorithm~\ref{alg: Our Method}.

        \subsubsection*{Exact Update}
        Suppose we are able to run DINGO with exact update. 
        For iterations in \labelcref{Case 2}, we obtain the following convergence result.
        
         \begin{theorem}[Convergence Under Case 2]\label{theorem: Case 2}
        	Suppose Assumption~\ref{assumption: Lipschitz Continuity} holds and that we run Algorithm~\ref{alg: Our Method}. 
            Then for all iterations $t$ in \labelcref{Case 2} we have $\|\bgg_{t+1}\|^2 \leq (1-2\tau_2\rho\theta) \|\bgg_t\|^2$ with constant
            \begin{equation}\label{eq: Case 2 Convergence}
            	\tau_2 = \frac{2(1-\rho)\phi^2\theta}{L(\ww_0)},
            \end{equation}
            where $L(\ww_0)$ is as in \eqref{eq: Lipschitz to Moral}, and $\rho, \theta$ and $\phi$ are as in Algorithm~\ref{alg: Our Method}. 
            Moreover, for iterations to exist in \labelcref{Case 2} it is necessary that $\theta\leq\sqrt{L(\ww_0)}/\phi$, which implies $0\leq 1-2\tau_2\rho\theta <1$.
        \end{theorem}
        
        In our experience, we have found that \labelcref{Case 2} does not occur frequently in practice. 
        It serves more of a theoretical purpose and is used to identify when \labelcref{Case 3} is required. 
        \labelcref{Case 2} may be thought of as a specific instance of \labelcref{Case 3}, in which $\mathcal{I}_t$ is empty. 
        However, it merits its own case, as in analysis it does not require additional assumptions to Assumption~\ref{assumption: Lipschitz Continuity} and in practice it may avoid an additional two communication rounds.

        \subsubsection*{Inexact Update}
        We now suppose that we run DINGO with inexact update.
        Recall, for a linear system $\AA\xx=\bb$, with non-singular square matrix $\AA$, the condition number of $\AA$, denoted as $\kappa(\AA)$, is defined by:
        \begin{equation}\label{eq: condition number}
            \kappa(\AA) 
            \defeq \max_{\ee,\bb\neq\zero} \bigg\{\frac{\|\AA^{-1}\ee\|}{\|\ee\|}\frac{\|\bb\|}{\|\AA^{-1}\bb\|}\bigg\}
            = \max_{\ee\neq\zero} \bigg\{\frac{\|\AA^{-1}\ee\|}{\|\ee\|}\bigg\}
            \max_{\xx\neq\zero} \bigg\{\frac{\|\AA\xx\|}{\|\xx\|}\bigg\}
            = \|\AA^{-1}\|\|\AA\|.
        \end{equation}
        Under the condition in \eqref{eq: Case 2 inexactness condition}, it follows from \eqref{eq: condition number}, as $\tilde{\HH}_{t,i}$ has full column rank, that
        \begin{equation*}
            \|\vv_{t,i}^{(2)} - \tilde{\HH}_{t,i}^\dagger\tilde{\bgg}_t\|
            \leq \varepsilon_i^{(2)} \kappa(\tilde{\HH}_{t,i}^T\tilde{\HH}_{t,i}) \|\tilde{\HH}_{t,i}^\dagger\tilde{\bgg}_t\|.
        \end{equation*}
        Then under Assumption~\ref{assumption: Lipschitz Continuity} we have
        \begin{equation}\label{eq: Case 2 inexactness condition corollary}
            \|\vv_{t,i}^{(2)} - \tilde{\HH}_{t,i}^\dagger\tilde{\bgg}_t\|
            \leq \varepsilon_i^{(2)}\bigg(\frac{K_i^2+\phi^2}{\phi^2}\bigg) \|\tilde{\HH}_{t,i}^\dagger\tilde{\bgg}_t\|.
        \end{equation}
        Using this upper bound, for iterations in \labelcref{Case 2*}, recall that this is the inexact variant of \labelcref{Case 2}, we obtain the following convergence result.
        
        \begin{theorem}[Convergence Under Case 2*]\label{theorem: Case 2 Inexact}
        	Suppose Assumption~\ref{assumption: Lipschitz Continuity} holds and that we run Algorithm~\ref{alg: Our Method} with inexact update such that the inexactness condition in \eqref{eq: Case 2 inexactness condition} holds. 
            Then for all iterations $t$ in \labelcref{Case 2*} we have $\|\bgg_{t+1}\|^2 \leq (1-2\tilde{\tau}_2\rho\theta) \|\bgg_t\|^2$ with constants
            \begin{equation}\label{eq: Case 2 Inexact Convergence}
            	\tilde{\tau}_2 = \frac{2(1-\rho)\theta}{L(\ww_0)c^2},
            	\quad
            	c = \frac{1}{\phi} \Bigg(1 + \frac{1}{m}\sum_{i=1}^{m} \varepsilon_i^{(2)} \frac{K_i^2+\phi^2}{\phi^2}\Bigg),
            \end{equation}
            where $L(\ww_0)$ is as in \eqref{eq: Lipschitz to Moral}, $\rho, \theta$ and $\phi$ are as in Algorithm~\ref{alg: Our Method}, $K_i$ are as in Assumption~\ref{assumption: Lipschitz Continuity}, and $\varepsilon_i^{(2)}$ are as in \eqref{eq: Case 2 inexactness condition}.
            Moreover, for iterations to exist in \labelcref{Case 2*} it is necessary that $\theta\leq c\sqrt{L(\ww_0)}$, which implies $0\leq 1-2\tilde{\tau}_2\rho\theta <1$.
        \end{theorem}
        
        Theorem~\ref{theorem: Case 2} gives a faster convergence rate than Theorem~\ref{theorem: Case 2 Inexact} as $\tau_2\geq\tilde{\tau}_2$, where $\tau_2$ and $\tilde{\tau}_2$ are as in Theorem~\ref{theorem: Case 2} and Theorem~\ref{theorem: Case 2 Inexact}, respectively.
        Reducing the inexactness error $\varepsilon_i^{(2)}$ leads to faster convergence in Theorem~\ref{theorem: Case 2 Inexact}.
        Moreover, if $\tilde{\HH}_{t,i}^\dagger\tilde{\bgg}_t$ is computed exactly, that is $\varepsilon_i^{(2)}=0$ for all $i=1,\ldots,m$, then Theorem~\ref{theorem: Case 2 Inexact} gives the same convergence as Theorem~\ref{theorem: Case 2}.

    \subsection{Analysis of Case 3 and Case 3*}\label{Section: Case 3}
        Now we turn to the final case, and analyze the convergence of iterations of DINGO that fall under \labelcref{Case 3} and \hbox{\labelcref{Case 3*}} under the condition of exact and inexact update, respectively.
        For such iterations, we make the following assumption about pseudo-inverse regularity of the full Hessian, rather than local pseudo-inverse regularity as in Assumption~\ref{assumption: Case 1 Gammas}.
        
        \begin{assumption}[Pseudo-Inverse Regularity \cite{Newton-MR}]\label{assumption: Case 3 Gamma}
            There exists a constant $\gamma\in(0,\infty)$ such that for all $\ww\in\bbR^d$ we have $\big\|\grad^2f(\ww)\pp\big\| \geq \gamma\big\|\pp\big\|$ for all $\pp\in\mathcal{R}\big(\grad^2f(\ww)\big)$ with $f$ as in \eqref{eq: finite sum}.
        \end{assumption}
        
        Additionally, we make the following assumption about the projection of the gradient onto the range of the Hessian.
        
        \begin{assumption}[Gradient-Hessian Null-Space Property \cite{Newton-MR}]\label{assumption: GHNSP}
            There exists a constant $\nu\in(0,1]$ such that
            \begin{equation*}
                \big\| (\UU_\ww^\perp)^T\grad f(\ww) \big\|^2
                \leq
                \frac{1-\nu}{\nu} \big\| \UU_\ww^T\grad f(\ww) \big\|^2,
            \end{equation*}
            for all $\ww\in\bbR^d$, where $\UU_\ww$ and $\UU_\ww^\perp$ denote arbitrary orthonormal bases for $\mathcal{R}\big(\grad^2f(\ww)\big)$ and its orthogonal complement, respectively, with $f$ as in \eqref{eq: finite sum}.
        \end{assumption}
        
        Assumption~\ref{assumption: GHNSP} implies that, as the iterations progress, the gradient will not become arbitrarily orthogonal to the range space of the Hessian matrix.
        As an example, take an under-determined least squares problem, where we have $f_i(\ww)=\|\AA_i\ww-\bb_i\|^2/2$ in \eqref{eq: finite sum}. 
        In this example, each $f_i$ is clearly not strongly convex; yet, Assumptions~\ref{assumption: Case 3 Gamma} and \ref{assumption: GHNSP} are satisfied with $\gamma = \sigma_\text{min}^2(\AA)/m$ and $\nu=1$, respectively, where $\AA=[\AA_1^T\cdots\AA_m^T]^T$ and $\sigma_\text{min}(\AA)$ is the smallest non-zero singular value of $\AA$.
        See \cite{Newton-MR} for a detailed discussion and many examples of Assumption~\ref{assumption: GHNSP}, in the case of $m=1$.
        
        Under Assumption~\ref{assumption: GHNSP}, we have both
        \begin{equation}\label{equation: Fred 2}
            \big\| \grad f(\ww) \big\|^2 \leq \frac{1}{\nu} \big\| \UU_\ww^T\grad f(\ww) \big\|^2
            \quad\text{and}\quad
            \big\| \grad f(\ww) \big\|^2 \geq \frac{1}{1-\nu} \big\| (\UU_\ww^\perp)^T\grad f(\ww) \big\|^2,
        \end{equation}
        for all $\ww\in\bbR^d$; see \cite{Newton-MR} for a proof. 
        Under Assumptions~\ref{assumption: Lipschitz Continuity}, \ref{assumption: Case 3 Gamma} and \ref{assumption: GHNSP}, we have
        \begin{equation}\label{eq: case 3 lower bound}
            \big\langle (\tilde{\HH}_{t,i}^T\tilde{\HH}_{t,i})^{-1}\HH_t\bgg_t,\HH_t\bgg_t \big\rangle
            \geq \frac{1}{K_i^2+\phi^2} \|\HH_t\bgg_t\|^2
            = \frac{1}{K_i^2+\phi^2} \|\HH_t\UU_{\ww_t}\UU_{\ww_t}^T\bgg_t\|^2 
            \geq \frac{\gamma^2\nu}{K_i^2+\phi^2} \|\bgg_t\|^2,
        \end{equation}
        for all iterations $t$ and all $i=1,\ldots,m$.
        We use this result in our convergence analysis of \labelcref{Case 3} and \labelcref{Case 3*}.

        \subsubsection*{Exact Update}
        Suppose we are able to run DINGO with exact update. 
        For iterations in \labelcref{Case 3}, we obtain the following convergence result.
        
        \begin{theorem}[Convergence Under Case 3]\label{theorem: Case 3}
        	Suppose Assumptions~\ref{assumption: Lipschitz Continuity}, \ref{assumption: Case 3 Gamma} and \ref{assumption: GHNSP} hold and that we run Algorithm~\ref{alg: Our Method}.
            Then for all iterations $t$ in \labelcref{Case 3} we have $\|\bgg_{t+1}\|^2 \leq (1-2\tau_3\rho\theta) \|\bgg_t\|^2$ with constants
            \begin{equation}\label{eq: Case 3 Convergence}
                \tau_3 = \frac{2(1-\rho)\theta}{L(\ww_0)c^2},
            	\quad
                c = \frac{1}{\phi} + \frac{1}{m} \bigg( \frac{1}{\phi} + \frac{\theta}{\gamma\sqrt{\nu}} \bigg) \sum_{i=1}^{m} \sqrt{\frac{K_i^2+\phi^2}{\phi^2}},
            \end{equation}
            where $L(\ww_0)$ is as in \eqref{eq: Lipschitz to Moral}, $\rho$, $\theta$ and $\phi$ are as in Algorithm~\ref{alg: Our Method}, $K_i$ are as in Assumption~\ref{assumption: Lipschitz Continuity}, $\gamma$ is as in Assumption~\ref{assumption: Case 3 Gamma}, and $\nu$ is as in Assumption~\ref{assumption: GHNSP}. 
            Moreover, for any choice of $\theta>0$ we have $0\leq 1-2\tau_3\rho\theta <1$.
        \end{theorem}
        
        Comparing Theorems~\ref{theorem: Case 2} and \ref{theorem: Case 3}, iterations of DINGO should have slower convergence if they are in \labelcref{Case 3} rather than \labelcref{Case 2}. 
        This is also observed empirically in Section~\ref{Section: Experiments}.
        Furthermore, an iteration in \labelcref{Case 3} would require two more communication rounds than if it were to stop in \labelcref{Case 1} or \labelcref{Case 2}.
        Therefore, one may wish to choose $\theta$ and $\phi$ appropriately to reduce the chances that iteration $t$ falls in \labelcref{Case 3} or that $|\mathcal{I}_t|$ is large.
        Under this consideration, Lemma~\ref{lem: necessary conditions} presents a necessary condition on a relationship between $\theta$ and $\phi$. 
        
        \begin{lemma}\label{lem: necessary conditions}
            Suppose $f_i$ in \eqref{eq: finite sum} is twice differentiable, for all $i=1,\ldots,m$, and that we run Algorithm~\ref{alg: Our Method}. 
            If $|\mathcal{I}_t|<m$ for some iteration $t$, as in \eqref{eq: set of Case 3 iteration indices}, then $\theta\phi \leq \|\HH_t\bgg_t\|/\|\bgg_t\|$. 
        \end{lemma}
        
        \cref{lem: necessary conditions} suggests that we should pick $\theta$ and $\phi$ so that their product, $\theta \phi$, is small. 
        Clearly, choosing smaller $\theta$ will increase the chance of an iteration of DINGO being in \labelcref{Case 1} or \labelcref{Case 2}. 
        Choosing smaller $\phi$ will preserve more curvature information of the Hessian $\HH_{t,i}$ in $\tilde{\HH}_{t,i}^\dagger$.
        On the other hand, choosing larger $\phi$ decreases the upper bound on the condition number of $\tilde{\HH}_{t,i}^T\tilde{\HH}_{t,i}$, namely $(K_i^2+\phi^2)/\phi^2$.
        This in turn increases the lower bound on the step-size in Theorems~\ref{theorem: Case 2 Inexact}, \ref{theorem: Case 3} and \ref{theorem: Case 3 Inexact}, and allows for a higher degree of inexactness $\varepsilon_i^{(3)}$ in Theorem~\ref{theorem: Case 3 Inexact}.
        
        We can finally present a unifying result on the overall worst-case \emph{linear convergence rate} of DINGO under exact update.
        
        \begin{corollary}[Overall Linear Convergence of DINGO Under Exact Update]\label{corollary: unify}
        	Suppose Assumptions~\ref{assumption: Lipschitz Continuity}, \ref{assumption: Case 1 Gammas}, \ref{assumption: Case 3 Gamma} and \ref{assumption: GHNSP} hold and that we run Algorithm~\ref{alg: Our Method}.
        	Then for all iterations $t$ we have $\|\bgg_{t+1}\|^2 \leq (1-2\tau\rho\theta) \|\bgg_t\|^2$ with $\tau=\min\{\tau_1,\tau_2,\tau_3\}$, where $\tau_1$, $\tau_2$ and $\tau_3$ are as in Theorems~\ref{theorem: Case 1}, \ref{theorem: Case 2} and \ref{theorem: Case 3}, respectively, and $\rho$ and $\theta$ are as in Algorithm~\ref{alg: Our Method}.
        \end{corollary}

        \subsubsection*{Inexact Update}
        We now suppose that we run DINGO with inexact update.
        As $\tilde{\HH}_{t,i}$ has full column rank, it follows from the condition in \eqref{eq: Case 3 inexactness condition} and from \eqref{eq: condition number} that
        \begin{equation*}
            \big\| \vv_{t,i}^{(3)} - (\tilde{\HH}_{t,i}^T\tilde{\HH}_{t,i})^{-1}\HH_t\bgg_t \big\|
            \leq \varepsilon_i^{(3)} \kappa(\tilde{\HH}_{t,i}^T\tilde{\HH}_{t,i})  
            \big\| (\tilde{\HH}_{t,i}^T\tilde{\HH}_{t,i})^{-1}\HH_t\bgg_t \big\|.
        \end{equation*}
        Then under Assumption~\ref{assumption: Lipschitz Continuity} we have
        \begin{equation}\label{eq: Case 3 inexactness condition corollary}
            \big\| \vv_{t,i}^{(3)} - (\tilde{\HH}_{t,i}^T\tilde{\HH}_{t,i})^{-1}\HH_t\bgg_t \big\|
            \leq \varepsilon_i^{(3)} \bigg(\frac{K_i^2+\phi^2}{\phi^2}\bigg) 
            \big\| (\tilde{\HH}_{t,i}^T\tilde{\HH}_{t,i})^{-1}\HH_t\bgg_t \big\|.
        \end{equation}
        Using this upper bound, for iterations in \labelcref{Case 3*}, recall that this is the inexact variant of \labelcref{Case 3}, we obtain the following convergence result.
        
        \begin{theorem}[Convergence Under Case 3*]\label{theorem: Case 3 Inexact}
        	Suppose Assumptions~\ref{assumption: Lipschitz Continuity}, \ref{assumption: Case 3 Gamma} and \ref{assumption: GHNSP} hold and that we run Algorithm~\ref{alg: Our Method} with inexact update such that the inexactness conditions in \eqref{eq: Case 2 inexactness condition} and \eqref{eq: Case 3 inexactness condition} hold with $\varepsilon_i^{(3)} < \sqrt{\phi^2/(K_i^2+\phi^2)}$ for all $i=1,\ldots,m$. 
            Then for all iterations $t$ in \labelcref{Case 3*} we have $\|\bgg_{t+1}\|^2 \leq (1-2\tilde{\tau}_3\rho\theta) \|\bgg_t\|^2$ with constants
            \begin{subequations}\label{eq: Case 3 Inexact Convergence}
                \begin{align}
                    \tilde{\tau}_3 &= \frac{2(1-\rho)\theta}{L(\ww_0)c^2}, \label{eq: Case 3 Inexact Convergence tau} \\
                    c &= \frac{1}{\phi} 
                    \Bigg( 1 + \frac{1}{m}\sum_{i=1}^{m} \varepsilon_i^{(2)} \frac{K_i^2+\phi^2}{\phi^2} \Bigg) \nonumber \\
                    &\quad+ \frac{1}{m} \sum_{i=1}^{m}
                    \Bigg(
                    \frac{1+\varepsilon_i^{(3)}(K_i^2+\phi^2)/\phi^2}{1-\varepsilon_i^{(3)}\sqrt{(K_i^2+\phi^2)/\phi^2}}
                    \Bigg)
                    \Bigg(
                    \frac{1}{\phi} \bigg(1+\varepsilon_i^{(2)}\frac{K_i^2+\phi^2}{\phi^2}\bigg) + \frac{\theta}{\gamma\sqrt{\nu}}
                    \Bigg) \sqrt{\frac{K_i^2+\phi^2}{\phi^2}}, \label{eq: Case 3 Inexact Convergence c}
                \end{align}
            \end{subequations}
            where $L(\ww_0)$ is as in \eqref{eq: Lipschitz to Moral}, 
            $\rho, \theta$ and $\phi$ are as in Algorithm~\ref{alg: Our Method}, 
            $K_i$ are as in Assumption~\ref{assumption: Lipschitz Continuity}, 
            $\gamma$ is as in Assumption~\ref{assumption: Case 3 Gamma}, 
            $\nu$ is as in Assumption~\ref{assumption: GHNSP},
            $\varepsilon_i^{(2)}$ are as in \eqref{eq: Case 2 inexactness condition},
            and $\varepsilon_i^{(3)}$ are as in \eqref{eq: Case 3 inexactness condition}.
            Moreover, for any choice of $\theta>0$ we have $0\leq 1-2\tilde{\tau}_3\rho\theta <1$.
        \end{theorem}
        
        Theorem~\ref{theorem: Case 3} gives a faster convergence rate than Theorem~\ref{theorem: Case 3 Inexact} as $\tau_3\geq\tilde{\tau}_3$, where $\tau_3$ and $\tilde{\tau}_3$ are as in Theorem~\ref{theorem: Case 3} and Theorem~\ref{theorem: Case 3 Inexact}, respectively.
        Reducing the inexactness error $\varepsilon_i^{(2)}$ or $\varepsilon_i^{(3)}$ leads to faster convergence in Theorem~\ref{theorem: Case 3 Inexact}.
        Moreover, if $\tilde{\HH}_{t,i}^\dagger\tilde{\bgg}_t$ and $(\tilde{\HH}_{t,i}^T\tilde{\HH}_{t,i})^{-1}\HH_t\bgg_t$ are computed exactly, that is $\varepsilon_i^{(2)}=0$ and $\varepsilon_i^{(3)}=0$ for all $i=1,\ldots,m$, then Theorem~\ref{theorem: Case 3 Inexact} gives the same convergence as Theorem~\ref{theorem: Case 3}.
        
        We can finally present a unifying result on the overall worst-case \emph{linear convergence rate} of DINGO under inexact update.
        
        \begin{corollary}[Overall Linear Convergence of DINGO Under Inexact Update]\label{corollary: unify inexact}
        	Suppose Assumptions~\ref{assumption: Lipschitz Continuity}, \ref{assumption: Case 1 Gammas}, \ref{assumption: Case 3 Gamma} and \ref{assumption: GHNSP} hold and that we run Algorithm~\ref{alg: Our Method} with inexact update such that Condition~\ref{condition: inexactness condition} holds with $\varepsilon_i^{(3)} < \sqrt{\phi^2/(K_i^2+\phi^2)}$ for all $i=1,\ldots,m$.
        	Then for all iterations $t$ we have $\|\bgg_{t+1}\|^2 \leq (1-2\tilde{\tau}\rho\theta) \|\bgg_t\|^2$ with $\tilde{\tau}=\min\{\tilde{\tau}_1,\tilde{\tau}_2,\tilde{\tau}_3\}$, where $\tilde{\tau}_1$, $\tilde{\tau}_2$ and $\tilde{\tau}_3$ are as in Theorems~\ref{theorem: Case 1 Inexact}, \ref{theorem: Case 2 Inexact} and \ref{theorem: Case 3 Inexact}, respectively, and $\rho$, $\theta$ and $\phi$ are as in Algorithm~\ref{alg: Our Method}.
        \end{corollary}
        
        From Corollary~\ref{corollary: unify} and Corollary~\ref{corollary: unify inexact}, DINGO can achieve $\|\bgg_t\|\leq\varepsilon$ with $\mathcal{O}(\log(\varepsilon)/(\tau\rho\theta))$ and $\mathcal{O}(\log(\varepsilon)/(\tilde{\tau}\rho\theta))$ communication rounds under exact and inexact update, respectively.
        Moreover, the terms $\tau$ and $\tilde{\tau}$ give a lower bound on the step-size under all exact and inexact cases, respectively, which can determine the maximum communication cost needed during line-search. 
        For example, knowing $\tilde{\tau}$ could determine the number of step-sizes used in backtracking line-search for DINGO in Section~\ref{Section: Experiments}.

\section{Experiments}\label{Section: Experiments}
    In this section, we examine the empirical performance of DINGO in comparison to several other distributed first-order and second-order optimization methods.
    In Section~\ref{section: Implementation Details}, we discuss the implementation details of DINGO.
    We then outline the optimization methods that we use in our evaluations in Section~\ref{sec: Optimisation Methods}.
    In Section~\ref{Section: Numerical Examples}, we consider simulations on model problems. 
    Our performance metric is communication rounds, since, as we have discussed in Section~\ref{Section: Introduction}, communications are often considered a major bottleneck.
    All methods are implemented to run on both CPU and GPU using Python and PyTorch.
    Code is available at \url{https://github.com/RixonC/DINGO}.

    \subsection{Implementation Details of DINGO}\label{section: Implementation Details}
        In this section, we discuss our implementation details for computing the update direction $\pp_t$ and step-size $\alpha_t$ of DINGO.

        \subsubsection*{Update Direction}
            Recall that the update direction $\pp_t$ involves the computation of approximations $\vv_{t,i}^{(1)}$, $\vv_{t,i}^{(2)}$ and $\vv_{t,i}^{(3)}$ of $\HH_{t,i}^\dagger\bgg_t$, $\tilde{\HH}_{t,i}^\dagger\tilde{\bgg}_t$ and $(\tilde{\HH}_{t,i}^T\tilde{\HH}_{t,i})^{-1}(\HH_t\bgg_t)$, respectively.
            Each of these present a unique type of least-squares problem.
            Many iterative least-squares solvers exist, as they are often designed to have specific properties.
            For our implementation of DINGO, we compute $\vv_{t,i}^{(1)}$, $\vv_{t,i}^{(2)}$ and $\vv_{t,i}^{(3)}$ using MINRES-QLP \cite{MINRES-QLP}, LSMR \cite{LSMR} and CG, respectively, with early termination.
            The reasoning behind these choices is presented below.
            
            \begin{itemize}
            \item \textbf{MINRES-QLP}: 
                Recall that MINRES \cite{MINRES} is a Krylov subspace method for iteratively solving $\argmin_{\xx} \|\AA\xx-\bb\|$, where $\AA$ is symmetric, and potentially indefinite and/or singular. 
                However, for inconsistent systems, it might not return the minimum-norm solution, $\xx=\AA^\dagger\bb$. 
                MINRES-QLP, however, is a variant of MINRES that not only is theoretically guaranteed to return the minimum-norm solution, but also it has been shown to perform better on ill-conditioned problems. 
                As a result, we use MINRES-QLP to compute $\HH_{t,i}^\dagger\bgg_t$.

            \item \textbf{LSMR}:
                To compute $\tilde{\HH}_{t,i}^\dagger\tilde{\bgg}_t$ we use LSMR, which is an iterative method for computing a solution to the following problems:
                \begin{equation*}
                    \argmin_{\xx} \|\AA\xx-\bb\|^2, \quad
                    \argmin_{\xx} \big(\|\AA\xx-\bb\|^2 + \lambda^2\|\xx\|^2\big), \quad
                    \argmin_{\xx} \|\xx\| \text{ subject to } \AA\xx=\bb,
                \end{equation*}
                where $\AA$ can have any rank and be either square or non-square, and $\lambda\geq 0$.
                As $\tilde{\HH}_{t,i}$ has full column rank, LSMR will iteratively compute the unique solution
                \begin{equation*}
                    \xx 
                    = (\tilde{\HH}_{t,i}^T\tilde{\HH}_{t,i})^{-1}\tilde{\HH}_{t,i}^T \tilde{\bgg}_t
                    = \tilde{\HH}_{t,i}^\dagger\tilde{\bgg}_t,
                \end{equation*}
                to the normal equation $ \tilde{\HH}_{t,i}^T\tilde{\HH}_{t,i}\xx = \tilde{\HH}_{t,i}^T \tilde{\bgg}_t$. 
                As
                $\| \tilde{\HH}_{t,i}\xx+\tilde{\bgg}_t \|^2 = \| \HH_{t,i}\xx+\bgg_t \|^2 + \phi^2\|\xx\|^2$,
                LSMR is well-suited for this problem.
                One may also consider using LSQR \cite{LSQR} to compute $\tilde{\HH}_{t,i}^\dagger\tilde{\bgg}_t$. 
                However, we choose LSMR as it is recommended for situations where the solver is terminated early as 
                it may be able to converge in significantly fewer iterations than LSQR, see \cite{LSMR}.

            \item \textbf{CG}:
                Arguably, the most well-known Krylov subspace method is CG. 
                It iteratively solves the problem $\AA\xx=\bb$ for $\xx$, where $\AA$ is both symmetric and positive definite.
                We use CG to find the approximation $\vv_{t,i}^{(3)}$ of $(\tilde{\HH}_{t,i}^T\tilde{\HH}_{t,i})^{-1}(\HH_t\bgg_t)$ as this will guarantee $\langle \vv_{t,i}^{(3)}, \HH_t\bgg_t \rangle > 0$, regardless of the number of CG iterations performed. 
                Computing $\vv_{t,i}^{(3)}$ as an approximation of $\tilde{\HH}_{t,i}^\dagger \big( (\tilde{\HH}_{t,i}^T)^\dagger \HH_t\bgg_t \big)$ using, for example, two instances of LSMR could be suited to other implementations of DINGO, while it might not satisfy $\langle \vv_{t,i}^{(3)}, \HH_t\bgg_t \rangle > 0$ for a given number of least-squares solver iterations. 
        \end{itemize}
        
        \subsubsection*{Line Search}
            For our experiments, we use backtracking line-search to choose the step-size $\alpha_t$. 
            Specifically, the driver broadcasts the update direction $\pp_t$ to all workers and then, by a reduce operation, each worker $i$ returns the vectors
            \begin{equation*}
                \grad f_i(\ww_t+\omega^0\pp_t), \grad f_i(\ww_t+\omega^1\pp_t), \ldots, \grad f_i(\ww_t+\omega^{k-1}\pp_t),
            \end{equation*}
            for some chosen $\omega\in(0,1)$ and maximum line-search iterations $k\geq1$.
            With this information, the driver chooses $\alpha_t$ to be the largest element of $\{\omega^0,\omega^1,\ldots,\omega^{k-1}\}$ such that \eqref{eq: Armijo-type line seach} holds, provided that $k$ is selected to be large enough so such $\alpha_t$ exists. 
            
            This line-search only requires two communication rounds. 
            However, increasing both $\omega$ and $k$ will improve the accuracy of $\alpha_t$, while also increasing the amount of data transmitted. 
            GIANT, in \cite{GIANT}, uses an analogous line-search technique on the objective value of $f$, which requires each worker to return $k$ scalars. 
            Whereas, for DINGO, we require each worker to return $k$ vectors in $\bbR^d$. 
            This method of line-search has an advantage for DINGO, where, in the subsequent iteration of Algorithm~\ref{alg: Our Method}, we can reduce communication rounds by circumventing the step where we distributively compute the full gradient.
            More efficient methods of line-search, for DINGO, will be explored in future work.

    \subsection{Optimization Methods}\label{sec: Optimisation Methods}
        Below we provide an outline of the compared methods and discuss their selected hyper-parameters.
        \begin{enumerate}
            \item \textbf{DINGO}. 
                In our experiments, unless otherwise stated, we set $\theta=10^{-4}$ and $\phi=10^{-6}$. 
                MINRES-QLP, LSMR and CG are limited to 50 iterations.
                For line-search, we use an Armijo line-search parameter, $\rho$ in \eqref{eq: Armijo-type line seach}, of $10^{-4}$ and we select the largest $\alpha_t\in\{1,2^{-1},2^{-2},\ldots,2^{-50}\}$ such that \eqref{eq: Armijo-type line seach} holds.
            \item \textbf{GIANT}.
                At iteration $t$, each worker $i$ approximately solves the linear system $\HH_{t,i}\xx = \bgg_{t}$ for $\xx$
                using CG. 
                We limit CG to 50 iterations. 
                We use backtracking line-search as in \cite{GIANT} to select the largest step-size in $\{1,2^{-1},2^{-2},\ldots,2^{-50}\}$ which passes and we use an Armijo line-search parameter of $10^{-4}$.
            \item \textbf{DiSCO}.
                Each iteration $t$ of DiSCO invokes a distributed \textit{preconditioned conjugate gradient} (PCG) method 
                that iteratively solves the linear system $\PP_t^{-1}\HH_t\vv_t = \PP_t^{-1}\bgg_t$ for $\vv_t$, 
                where $\PP_t$ is a preconditioning matrix.
                For consistency, we will not employ preconditioning in our examined methods and thus we set $\PP_t=\eye$ for all iterations $t$.
                We also restrict the number of distributed PCG iterations to 50.
            \item \textbf{InexactDANE}.
                At iteration $t$, each worker $i$ inexactly solves the sub-problem
                \begin{equation*}
                    \argmin_{\ww}
                    \Big[
                        f_i(\ww) - \big(\nabla f_i(\ww_{t-1}) - \eta\nabla f(\ww_{t-1})\big)^T\ww + \frac{\mu}{2}\|\ww-\ww_{t-1}\|^2
                    \Big],
                \end{equation*}
                where $\eta,\mu\geq0$ are selected hyper-parameters. 
                We use SVRG as a local solver for this sub-problem, where we limit SVRG to 50 iterations and report the best learning rate from
                $\{10^{-6},10^{-5},\ldots,10^{6}\}$.
                SVRG was used as a local solver for InexactDANE in \cite{AIDE}.
                We also set $\eta=1$ and $\mu=0$ in all experiments as we found these resulted in best performance.
                In fact, these parameters often yielded high performance in \cite{DANE} as well.
            \item \textbf{AIDE}.
                Each iteration of AIDE invokes the InexactDANE method on a modified version of the objective function \eqref{eq: finite sum}
                and involves a catalyst acceleration parameter $\tau\geq0$.
                We have each iteration of AIDE invoke one iteration of InexactDANE.
                We will use the same parameters as in the stand-alone InexactDANE method and report the best
                $\tau\in\{10^{-6},10^{-5},\ldots,10^{6}\}$.
            \item \textbf{Asynchronous Stochastic Gradient Descent} (Async-SGD, \cite{SGD}).
                In our experiments we report the best learning rate from $\{10^{-6},10^{-5},\ldots,10^{6}\}$.
                Each worker uses a mini-batch of size $n/(5m)$.
                Due to staleness in the update directions, especially when $m$ is large, we often report a smaller learning rate for Async-SGD than in Sync-SGD.
            \item \textbf{Synchronous Stochastic Gradient Descent} (Sync-SGD, \cite{SGD}).
                Like Async-SGD, we report the best learning rate from $\{10^{-6},10^{-5},\ldots,10^{6}\}$.
                Each worker uses a mini-batch of size $n/(5m)$.
        \end{enumerate}
        
        \begin{table}
    \centering
    \caption{Details of the datasets used for softmax regression problems.}
    \vspace{10pt}
    \begin{tabular}{lccccc} 
        \toprule
        & Train Size & Test Size & Features & Classes & Problem Dimension \\ 
        &            &           & ($p$)    & ($C$)   & ($d=p(C-1)$)      \\ 
        \midrule
        \textbf{CIFAR10}       & 50,000  & 10,000 & 3,072 & 10 & 27,648 \\ 
        \textbf{EMNIST Digits} & 240,000 & 40,000 & 784   & 10 & 7,056  \\
        \bottomrule
    \end{tabular}
    \label{table: softmax datasets}
\end{table}

    \subsection{Simulations on Model Problems}\label{Section: Numerical Examples}
        In this section, we compare the performance of the optimization methods on a variety of softmax cross-entropy minimization problems with regularization. 
        In all experiments we consider \eqref{eq: finite sum} with \eqref{eq: finite sum dist fi}, where the sets $S_1,\ldots,S_m$ randomly partition the index set $\{1,\ldots,n\}$, with each having equal size $s=n/m$.
        We show the performance of the optimization methods applied to the softmax regression problem on the CIFAR10 dataset in Figure~\ref{fig: softmax CIFAR10} and on the EMNIST Digits dataset in Figures~\ref{fig: softmax MNIST} and \ref{fig: softmax EMNIST time}.
        The properties of these datasets are provided in Table~\ref{table: softmax datasets}.
        CIFAR10 presents a situation where $s<d$ in all experiments.
        On the other hand, EMNIST Digits has a large number of samples $n$ and we have $s>d$ in all experiments.
        
        DiSCO has consistent performance, regardless of the number of workers, due to the distributed PCG algorithm. 
        This essentially allows DiSCO to perform Newton's method over the full dataset. 
        This is unnecessarily costly, in terms of communication rounds, when $s$ is reasonably large. 
        Thus we see it perform comparatively poorly in Plots \ref{fig: softmax CIFAR10 10}, \ref{fig: softmax CIFAR10 100} and \ref{fig: softmax CIFAR10 1000}, and Figure~\ref{fig: softmax MNIST}. 
        DiSCO outperforms GIANT and DINGO in Plot~\ref{fig: softmax CIFAR10 10000}. 
        This is likely because the local directions ($-\HH_{t,i}^{-1}\bgg_t$ and $\pp_{t,i}$ for GIANT and DINGO, respectively) give poor updates as they are calculated using very small subsets of data. 
        As an example, in Plot~\ref{fig: softmax CIFAR10 10000} each worker has access to only five data points, while $d=27648$.
        
        A significant advantage of DINGO to InexactDane, AIDE, Async-SGD and Sync-SGD is that it is relatively easy to tune hyper-parameters. 
        Namely, making bad choices for $\rho$, $\theta$ and $\phi$ in Algorithm~\ref{alg: Our Method} will give sub-optimal performance; however, it is still theoretically guaranteed to strictly decrease the norm of the gradient. 
        In contrast, some choices of hyper-parameters in InexactDane, AIDE, Async-SGD and Sync-SGD will cause divergence and these choices can be problem specific.
        Moreover, these methods can be very sensitive to the chosen hyper-parameters with some being very difficult to select. 
        For example, the acceleration parameter $\tau$ in AIDE was found to be difficult and time consuming to tune and the performance of AIDE was sensitive to it; notice the variation in selected $\tau$ in Figure~\ref{fig: softmax CIFAR10}.
        This difficulty was also observed in \cite{GIANT, AIDE}. 
        We found that simply choosing $\rho$, $\theta$ and $\phi$ to be small, in DINGO, gave high performance.
        
        In Plot~\ref{fig: softmax MNIST 16 DI} we demonstrate the effect of choosing unnecessarily large values of $\theta$ for DINGO.
        In this experiment, each iteration of DINGO is in \labelcref{Case 1} when $\theta=10^{-4}$, \labelcref{Case 1} and \labelcref{Case 3} occur when $\theta=1$, and each iteration is in \labelcref{Case 3} when $\theta=100$. 
        We maintain a step-size of $1$ when $\theta=1$ and $\theta=100$ and we obtain similar convergence in the objective value for all three values of $\theta$. 
        In the experiment of Plot~\ref{fig: softmax MNIST 16 DI}, we obtain the exact same convergence for all values $\theta\leq 10^{-1}$.
        
        In Figure~\ref{fig: softmax EMNIST time} we compare run-times in a distributed environment running over six Amazon Elastic Compute Cloud instances via Amazon Web Services (AWS). 
        These instances are spread globally to highlight the effects of communication latency on run-time.
        Namely, they are located in Ireland, Ohio, Oregon, Singapore, Sydney and Tokyo, with the driver node taking the Ohio instance and the five worker nodes taking the others.
        In terms of run-time, DINGO is competitive with the other second-order methods, all of which outperform SGD.
        DINGO is able to achieve this while using significantly fewer communication rounds than the other methods.
        The effect of communication latency can be particularly observed in Table~\ref{table: iterations}, where we compare the number of iterations completed in one hour when the driver and all five worker machines are running locally on one node and when they are instead distributed over AWS.
        DINGO is able to effectively utilize the increased computational resources of a distributed setting, hence it's performance increased when moving workers from a single node to the distributed environment on AWS.
        
        \begin{figure}%
    \centering
    \hspace*{-12pt}\subfigure{%
    \includegraphics[height=22pt]{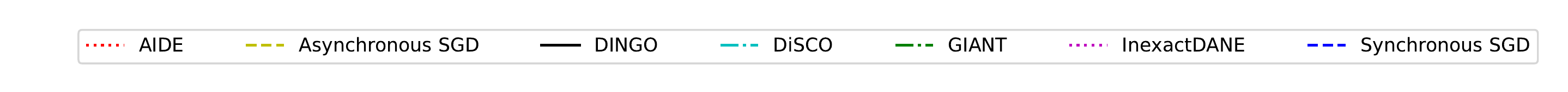}}%
    \\
    \addtocounter{subfigure}{-1}
    \subfigure[10 Workers]{%
    \label{fig: softmax CIFAR10 10}%
    \includegraphics[height=10cm]{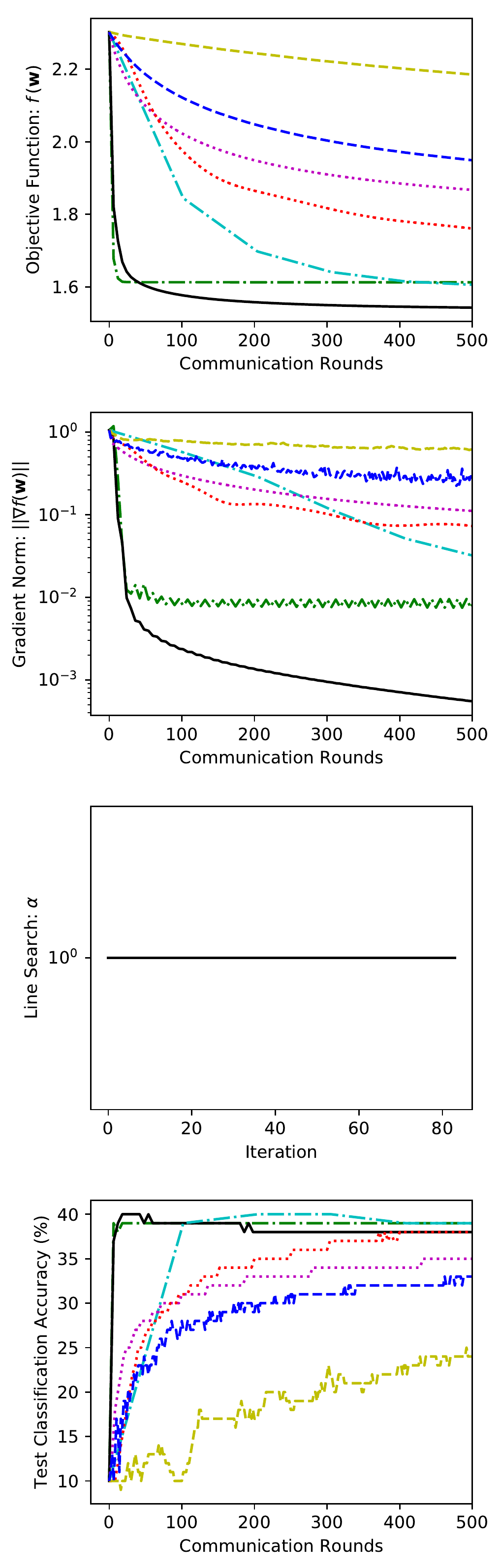}}%
    \qquad
    \subfigure[100 Workers]{%
    \label{fig: softmax CIFAR10 100}%
    \includegraphics[height=10cm]{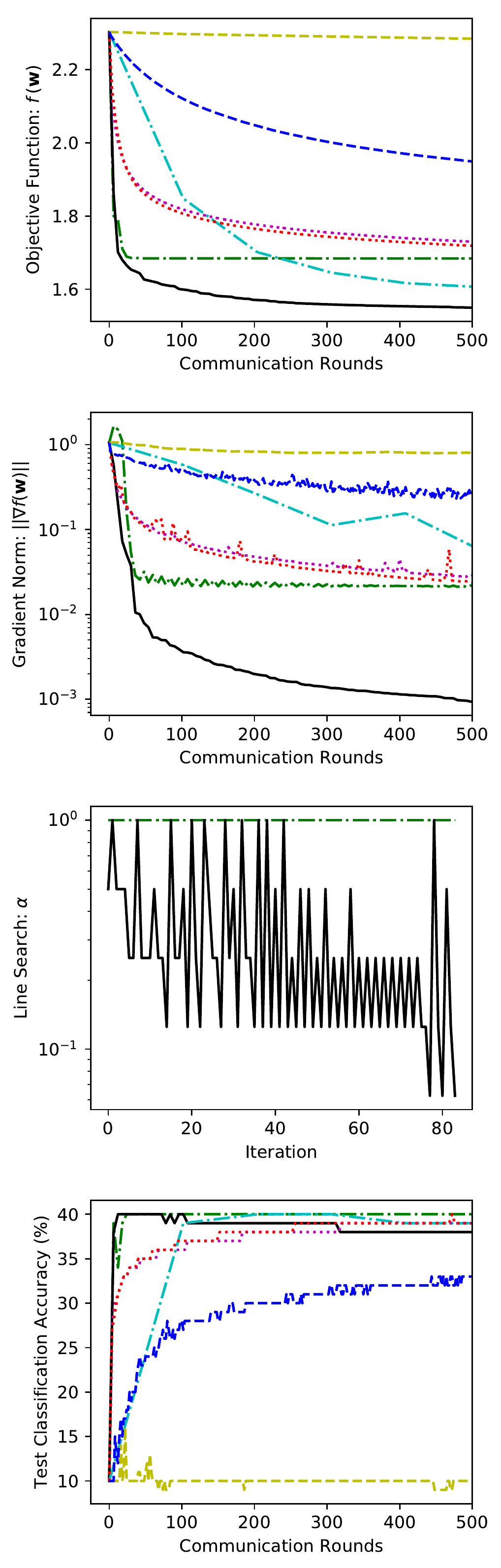}}%
    \qquad
    \subfigure[1000 Workers]{%
    \label{fig: softmax CIFAR10 1000}%
    \includegraphics[height=10cm]{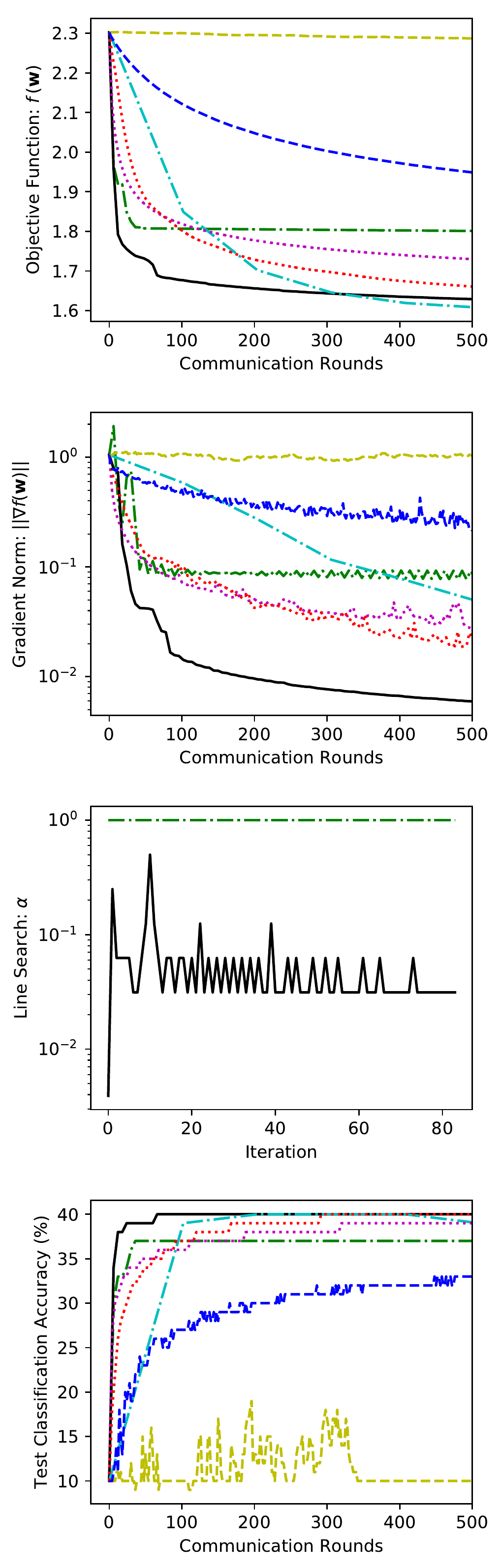}}%
    \qquad
    \subfigure[10000 Workers]{%
    \label{fig: softmax CIFAR10 10000}%
    \includegraphics[height=10cm]{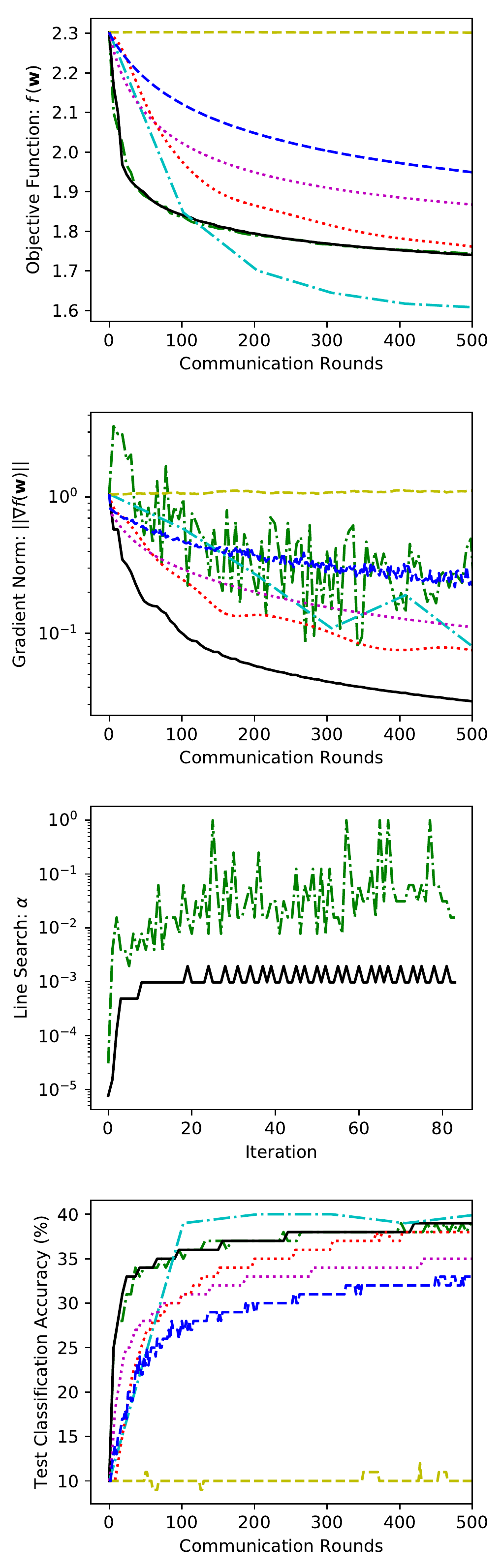}}%
    \caption{Softmax regression problem on the CIFAR10 dataset. 
    All algorithms are initialized at $\ww_0 = \zero$.
    In all plots, Sync-SGD has a learning rate of $10^{-2}$.
    Async-SGD has a learning rate of: $10^{-3}$ in \ref{fig: softmax CIFAR10 10}, 
    $10^{-4}$ in \ref{fig: softmax CIFAR10 100} and \ref{fig: softmax CIFAR10 1000}, 
    and $10^{-5}$ in \ref{fig: softmax CIFAR10 10000}.
    SVRG (for InexactDANE and AIDE) has a learning rate of: 
    $10^{-3}$ in \ref{fig: softmax CIFAR10 10} and \ref{fig: softmax CIFAR10 10000},
    and $10^{-2}$ in \ref{fig: softmax CIFAR10 100} and \ref{fig: softmax CIFAR10 1000}.
    AIDE has $\tau=100$ in \ref{fig: softmax CIFAR10 10} and \ref{fig: softmax CIFAR10 10000}, 
    $\tau=1$ in \ref{fig: softmax CIFAR10 100}, and 
    $\tau=10$ in \ref{fig: softmax CIFAR10 1000}.
    }
    \label{fig: softmax CIFAR10}
\end{figure}
        \begin{figure}%
    \centering
    \hspace*{-12pt}\subfigure{%
    \includegraphics[height=22pt]{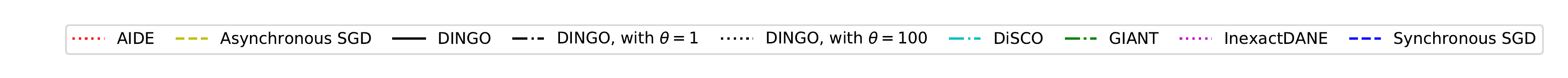}}%
    \\
    \addtocounter{subfigure}{-1}
    \subfigure[8 Workers]{%
    \label{fig: softmax MNIST 8}%
    \includegraphics[height=10cm]{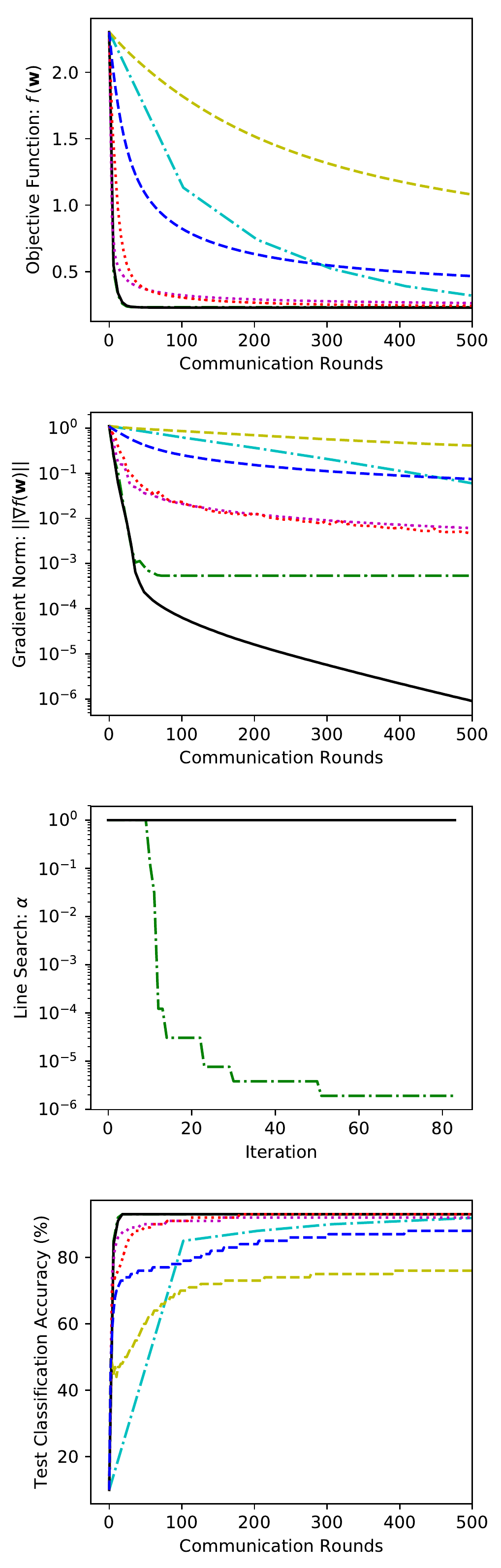}}%
    \qquad
    \subfigure[16 Workers]{%
    \label{fig: softmax MNIST 16}%
    \includegraphics[height=10cm]{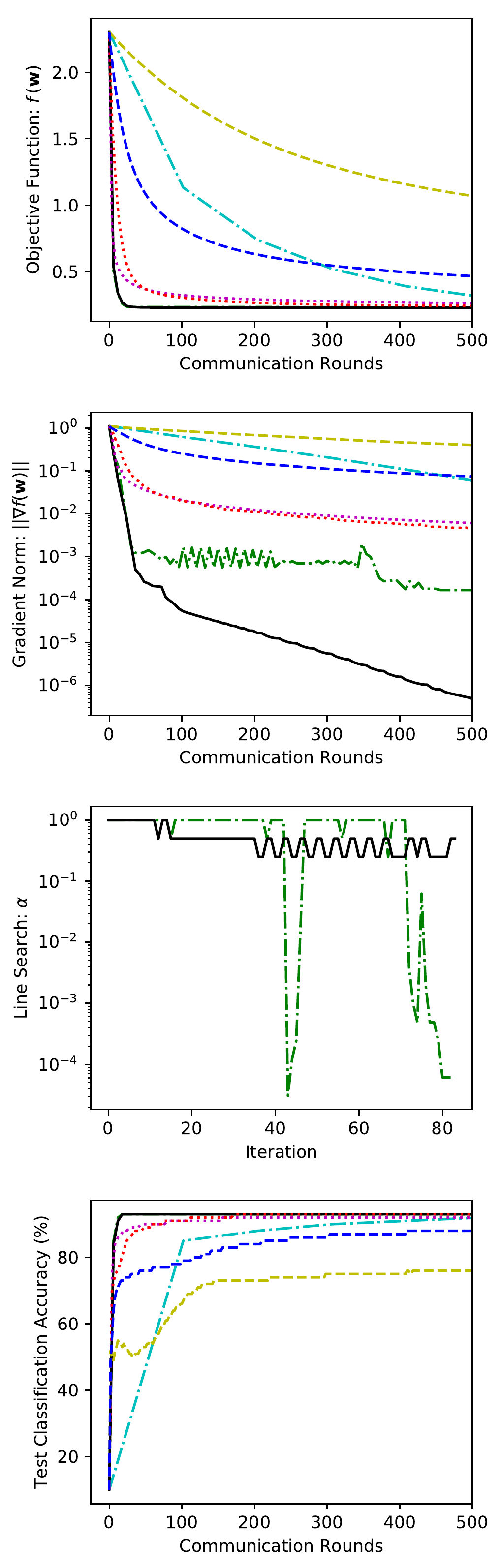}}%
    \qquad
    \subfigure[32 Workers]{%
    \label{fig: softmax MNIST 32}%
    \includegraphics[height=10cm]{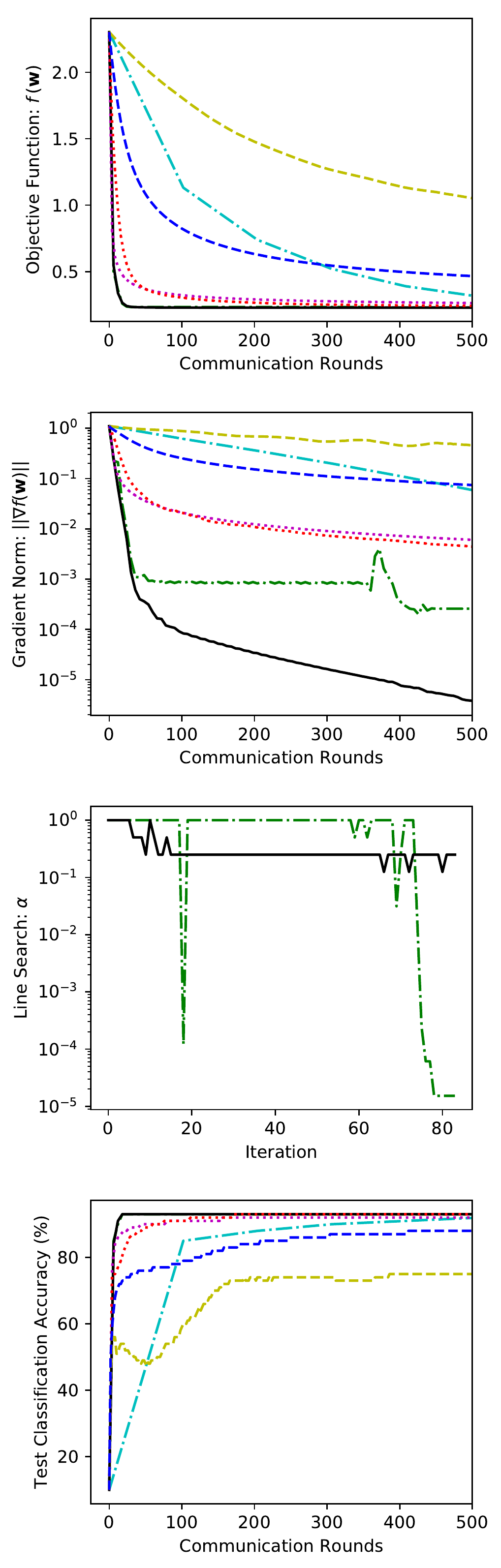}}%
    \qquad
    \subfigure[16 Workers]{%
    \label{fig: softmax MNIST 16 DI}%
    \includegraphics[height=10cm]{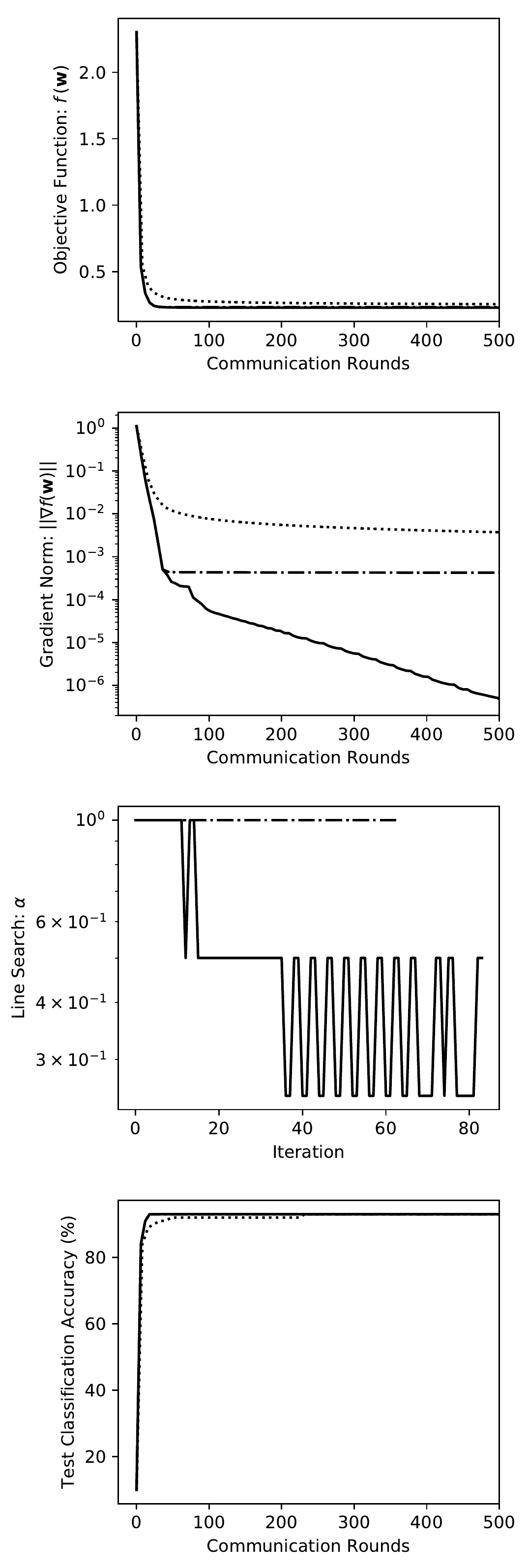}}%
    \caption{Softmax regression problem on the EMNIST Digits dataset. 
    In Plots \ref{fig: softmax MNIST 8}, \ref{fig: softmax MNIST 16} and \ref{fig: softmax MNIST 32}: 
    Async-SGD, Sync-SGD and SVRG have a learning rate of $10^{-2}$, $10^{-1}$ and $10^{-1}$, respectively, 
    and AIDE has $\tau=1$. 
    In Plot \ref{fig: softmax MNIST 16 DI} we compare the performance of DINGO for three different values of $\theta$, namely $10^{-4}$, $1$ and $100$. 
    In this plot, similar convergence in the objective value is obtained, while DINGO with $\theta=10^{-4}$ achieves a significantly faster reduction in the norm of the gradient.
    }
    \label{fig: softmax MNIST}
\end{figure}
        \begin{figure}%
    \centering
    \hspace*{-12pt}\subfigure{%
    \includegraphics[height=20pt]{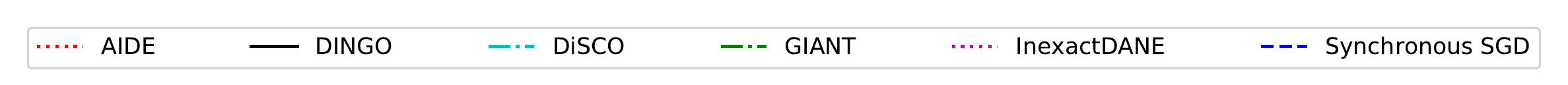}}%
    \\
    \addtocounter{subfigure}{-1}
    \subfigure{%
    \includegraphics[height=3.5cm]{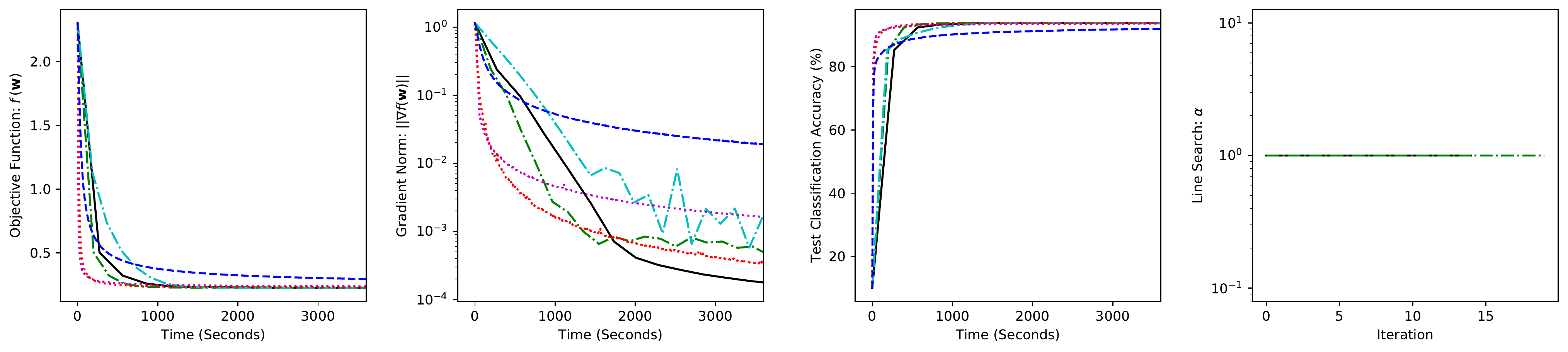}}%
    \caption{Run-time on Softmax regression problem on the EMNIST Digits dataset over AWS.
    Sync-SGD and SVRG both have a learning rate of $10^{-1}$ and AIDE has $\tau=1$.
    }
    \label{fig: softmax EMNIST time}
\end{figure}
        \begin{table}[t]
    \caption{Number of iterations completed in one hour when the driver and all five worker machines are running either on one node or on their own instances in the distributed computing environment over AWS.
    }\label{table: iterations}
    \vspace{10pt}
    \centering
    \begin{tabular}{@{} lccc @{}}
        \toprule
        & \multirow{2}{4.3cm}{\centering Number of Iterations \\ (Running on one Node)} 
        & \multirow{2}{4.2cm}{\centering Number of Iterations \\ (Running over AWS)} 
        & \multirow{2}{2.9cm}{\centering Change When \\ Going to AWS} \\ \\
        \midrule
        \textbf{DINGO}                   & $ 8    $ & $ 12   $ & $ +50\% $ \\
        \textbf{GIANT} \cite{GIANT}      & $ 11   $ & $ 18   $ & $ +64\% $ \\
        \textbf{DiSCO} \cite{DiSCO}      & $ 23   $ & $ 19   $ & $ -17\% $ \\
        \textbf{InexactDANE} \cite{AIDE} & $ 593  $ & $ 486  $ & $ -18\% $ \\
        \textbf{AIDE} \cite{AIDE}        & $ 601  $ & $ 486  $ & $ -19\% $ \\
        \textbf{Sync-SGD} \cite{SGD}     & $ 6003 $ & $ 1187 $ & $ -80\% $ \\
        \bottomrule
    \end{tabular}
\end{table}

\section{Conclusions and Future Work}\label{Section: conclusion}
    In this paper, we developed DINGO, which is a communication efficient Newton-type algorithm for optimization of an average of functions in a distributed computing environment.
    Similar second-order methods compromise between application range and ease of implementation. 
    DINGO is applicable to problems beyond (strongly) convex and does not impose any specific form on the underlying functions, making it suitable for many distributed optimization applications with an arbitrary distribution of data across the network.
    In addition, DINGO is simple in implementation as the underlying sub-problems are simple linear least-squares. 
    The few hyper-parameters are easy to tune and we theoretically showed a linear convergence-rate reduction in the gradient norm is guaranteed, regardless of the selected hyper-parameters. 
    Through empirical experiments, we demonstrated the effectiveness, stability and versatility of our method compared to other relevant algorithms.
    
    The following is left for future work.
    First, finding more efficient methods of line-search, for practical implementations of DINGO, than backtracking line-search. 
    As discussed previously, backtracking line-search may require the transmission of a large amount of data over the network, while still only requiring two communication rounds per iteration.
    Second, considering modifications to DINGO that prevent convergence to a local maximum/saddle point in non-invex problems.

    \subsubsection*{Acknowledgments}
    Both authors gratefully acknowledge the generous support by the Australian Research Council Centre of Excellence for Mathematical \& Statistical Frontiers (ACEMS). Fred Roosta was partially supported by the ARC DECRA Award (DE180100923). Part of this work was done while Fred Roosta was visiting the Simons Institute for the Theory of Computing. This material is based on research partially sponsored by DARPA and the Air Force Research Laboratory under agreement number FA8750-17-2-0122. The U.S. Government is authorized to reproduce and distribute reprints for Governmental purposes notwithstanding any copyright notation thereon. The views and conclusions contained herein are those of the authors and should not be interpreted as necessarily representing the official policies or endorsements, either expressed or implied, of DARPA and the Air Force Research Laboratory or the U.S. Government.

\bibliographystyle{unsrt}
\bibliography{bibliography}
    
\newpage
\appendix 
\section{Proofs of Results From Section~\ref{Section: Theoretical Analysis}}\label{appendix: proofs}
    In this section, we provide proofs of our results from Section~\ref{Section: Theoretical Analysis}.

    \subsection{Theorem \ref{theorem: Case 1}}\label{section: proof of case 1}
        \begin{proof}[Proof of Theorem~\ref{theorem: Case 1}]
            Suppose iteration $t$ is in \labelcref{Case 1}.
            For this iteration, we use the update direction $\pp_t = \sum_{i=1}^{m}\pp_{t,i}/m$, with $\pp_{t,i}=-\HH_{t,i}^{\dagger}\bgg_t$.
            Using Assumption~\ref{assumption: Case 1 Gammas} we obtain the following upper bound on $\|\pp_t\|$:
            \begin{equation*}
                \|\pp_t\|
                = \frac{1}{m} \Bigg\|\sum_{i=1}^{m}\pp_{t,i}\Bigg\|
                \leq \frac{1}{m}\sum_{i=1}^{m}\|\pp_{t,i}\|
                \leq \frac{1}{m}\sum_{i=1}^{m}\frac{1}{\gamma_i}\|\bgg_t\|
                = \frac{1}{\gamma}\|\bgg_t\|,
            \end{equation*}
            where $\gamma$ is as in \eqref{eq: Case 1 Convergence}.
            This and \eqref{eq: Moral-Smoothness Corollary} imply
            \begin{equation}\label{eq: exact update case 1 alpha inequality 1}
                \big\|\grad f(\ww_t+\alpha\pp_t)\big\|^2
                \leq \|\bgg_t\|^2 + 2\alpha\big\langle \pp_t,\HH_t\bgg_t \big\rangle + \frac{\alpha^2L(\ww_0)}{\gamma^2}\|\bgg_t\|^{2},
            \end{equation}
            for all $\alpha\geq0$.
            For all $\alpha\in(0,\tau_1]$, where $\tau_1$ is as in \eqref{eq: Case 1 Convergence}, we have
            \begin{equation*}
                \frac{\alpha^2L(\ww_0)}{\gamma^2}\|\bgg_t\|^{2} 
                \leq 2\alpha(1-\rho)\theta\|\bgg_t\|^2,
            \end{equation*}
            and as $\langle \pp_t,\HH_t\bgg_t \rangle \leq -\theta\|\bgg_t\|^2$, since this iteration $t$ is in \labelcref{Case 1}, we obtain
            \begin{equation*}
                \frac{\alpha^2L(\ww_0)}{\gamma^2}\|\bgg_t\|^{2} 
                \leq 2\alpha(\rho-1)\langle\pp_t,\HH_t\bgg_t\rangle.
            \end{equation*}
            From this and \eqref{eq: exact update case 1 alpha inequality 1} we have
            \begin{equation*}
                \big\|\grad f(\ww_t+\alpha\pp_t)\big\|^2
                \leq \|\bgg_t\|^2 + 2\alpha\langle\pp_t,\HH_t\bgg_t\rangle 
                + \frac{\alpha^2L(\ww_0)}{\gamma^2}\|\bgg_t\|^{2}
                \leq \|\bgg_t\|^2 + 2\alpha\rho\langle\pp_t,\HH_t\bgg_t\rangle,
            \end{equation*}
            for all $\alpha\in(0,\tau_1]$.
            Therefore, line-search \eqref{eq: Armijo-type line seach} will pass for some step-size $\alpha_t \geq \tau_1$ and $\| \bgg_{t+1} \|^2 \leq (1-2\tau_1\rho\theta) \|\bgg_t\|^2$.

            From $\langle \pp_t,\HH_t\bgg_t \rangle \leq -\theta\|\bgg_t\|^2$ and \eqref{eq: exact update case 1 alpha inequality 1} we have
            \begin{equation}\label{eq: exact update case 1 alpha inequality 2}
                \big\|\grad f(\ww_t+\alpha\pp_t)\big\|^2
                \leq \|\bgg_t\|^2 - 2\alpha\theta\|\bgg_t\|^2 + \frac{\alpha^2L(\ww_0)}{\gamma^2}\|\bgg_t\|^{2},
            \end{equation}
            for all $\alpha\geq0$.
            The right-hand side of \eqref{eq: exact update case 1 alpha inequality 2} is minimised when $\alpha=\theta\gamma^2/L(\ww_0)$ and has a minimum value of $\big(1-\theta^2\gamma^2/L(\ww_0)\big)\|\bgg_t\|^{2}$. 
            This implies $\theta\leq\sqrt{L(\ww_0)}/\gamma$.
            Therefore, 
            \begin{equation*}
                2\tau_1\rho\theta 
                = \frac{4\rho(1-\rho)\gamma^2\theta^2}{L(\ww_0)}
                \leq 4\rho(1-\rho)
                \leq 1,
            \end{equation*}
            which implies $0\leq 1-2\tau_1\rho\theta <1$.
        \end{proof}

    \subsection{Lemma \ref{Lemma: always case 1}}
        \begin{proof}[Proof of Lemma~\ref{Lemma: always case 1}]
            We have
            \begin{equation}\label{eq: always case 1}
                \|\bgg_t\|^2 - \bgg_t^T\HH_{t,i}^{-1}\HH_{t}\bgg_t
                = \bgg_t^T\HH_{t,i}^{-1}(\HH_{t,i}-\HH_t)\bgg_t
                \leq \varepsilon_i\|\bgg_t\|\|\HH_{t,i}^{-1}\bgg_t\|
                \leq \frac{\varepsilon_i}{\gamma_i}\|\bgg_t\|^2,
            \end{equation}
            for all iterations $t$ and all $i=1,\ldots,m$. 
            Therefore, if $\sum_{i=1}^{m}(1-\varepsilon_i/\gamma_i)/m \geq \theta$ then \eqref{eq: always case 1} implies
            \begin{equation*}
                \bigg\langle \frac{1}{m}\sum_{i=1}^{m}\HH_{t,i}^{-1}\bgg_t, \HH_t\bgg_t \bigg\rangle
                = \frac{1}{m}\sum_{i=1}^{m}\bgg_t^T\HH_{t,i}^{-1}\HH_{t}\bgg
                \geq \frac{1}{m}\sum_{i=1}^{m}\bigg(1-\frac{\varepsilon_i}{\gamma_i}\bigg)\|\bgg_t\|^2
                \geq \theta\|\bgg_t\|^2,
            \end{equation*}
            for all $t$, which implies all iterations $t$ are in \labelcref{Case 1}.
        \end{proof}

    \subsection{Theorem \ref{theorem: Case 1 Inexact}}\label{section: proof of case 1 Inexact}
        \begin{proof}[Proof of Theorem~\ref{theorem: Case 1 Inexact}]
            Suppose iteration $t$ is in \labelcref{Case 1*}.
            For this iteration, we use the update direction $\pp_t = \sum_{i=1}^{m}\pp_{t,i}$, where $\pp_{t,i}=-\vv_{t,i}^{(1)}$ and $\vv_{t,i}^{(1)}$ is an approximation of $\HH_{t,i}^\dagger\bgg_t$ satisfying the inexactness condition in \eqref{eq: Case 1 inexactness condition}.
            If we write $\bgg_t = \xx + \yy$, where $\xx\in\mathcal{R}(\HH_{t,i})=\mathcal{N}(\HH_{t,i})^\perp$ and $\yy\in\mathcal{N}(\HH_{t,i})$, then
            \begin{equation*}
                \|\HH_{t,i}^2\vv_{t,i}^{(1)} - \HH_{t,i}\bgg_t\| 
                \geq \gamma_i\|\HH_{t,i}\vv_{t,i}^{(1)} - \xx\|
                \geq \gamma_i\big(\|\HH_{t,i}\vv_{t,i}^{(1)}\| - \|\xx\|\big)
                \geq \gamma_i\big(\|\HH_{t,i}\vv_{t,i}^{(1)}\| - \|\bgg_t\|\big),
            \end{equation*}
            for all $i=1,\ldots,m$, where $\gamma_i$ is as in Assumption~\ref{assumption: Case 1 Gammas}.
            This result, \eqref{eq: Case 1 inexactness condition} and Assumptions~\ref{assumption: Lipschitz Continuity} and \ref{assumption: Case 1 Gammas} imply
            \begin{equation*}
                \gamma_i^2\|\vv_{t,i}^{(1)}\| 
                \leq \gamma_i\|\HH_{t,i}\vv_{t,i}^{(1)}\|
                \leq \gamma_i\|\bgg_t\| + \varepsilon_i^{(1)}\|\HH_{t,i}\bgg_t\|
                \leq (\gamma_i+\varepsilon_i^{(1)}K_i)\|\bgg_t\|,
            \end{equation*}
            for all $i=1,\ldots,m$.
            Therefore,
            \begin{equation*}
                \|\pp_t\| 
                \leq \frac{1}{m}\sum_{i=1}^{m} \frac{\gamma_i+\varepsilon_i^{(1)}K_i}{\gamma_i^2}\|\bgg_t\| 
                = \frac{1}{\gamma}\|\bgg_t\|,
            \end{equation*}
            where $\gamma$ is as in \eqref{eq: Case 1 Inexact Convergence}.
            From this and an analogous argument to that in the proof of Theorem~\ref{theorem: Case 1}, for all $\alpha\in(0,\tilde{\tau}_1]$, where $\tilde{\tau}_1$ is as in \eqref{eq: Case 1 Inexact Convergence}, we have
            \begin{equation*}
                \big\|\grad f(\ww_t+\alpha\pp_t)\big\|^2
                \leq \|\bgg_t\|^2 + 2\alpha\rho\langle\pp_t,\HH_t\bgg_t\rangle.
            \end{equation*}
            Therefore, line-search \eqref{eq: Armijo-type line seach} will pass for some $\alpha_t \geq \tilde{\tau}_1$ and $\| \bgg_{t+1} \|^2 \leq (1-2\tilde{\tau}_1\rho\theta) \|\bgg_t\|^2$.
            Moreover, from an analogous argument to that in the proof of Theorem~\ref{theorem: Case 1}, we have $\theta\leq\sqrt{L(\ww_0)}/\gamma$, which implies $0\leq 1-2\tilde{\tau}_1\rho\theta <1$.
        \end{proof}

    \subsection{Theorem \ref{theorem: Case 2}}\label{section: proof of case 2}
        \begin{lemma}\label{lem: pseudoinverse norm bound}
            Let $\AA_1\in\mathbb{R}^{m\times n}$ and $\AA_2\in\mathbb{R}^{n\times n}$ for arbitrary positive integers $m$ and $n$.
            Suppose that $\AA_2$ is non-singular.
            If we let
            \begin{equation*}
                \AA = \begin{bmatrix} \AA_1 \\ \AA_2 \end{bmatrix} \in \bbR^{(m+n)\times n},
            \end{equation*}
            then $\|\AA^\dagger\| \leq \|\AA_2^{-1}\|$.
        \end{lemma}
        \begin{proof}
            Recall that
            \begin{equation*}
                \|\AA^\dagger\| = \max_{\xx\neq\zero} \frac{\|\AA^\dagger\xx\|}{\|\xx\|}.
            \end{equation*}
            For any such $\xx$ we may write $\xx=\yy+\zz$ with $\yy\in\mathcal{R}(\AA)$ and $\zz\in\mathcal{R}(\AA)^\perp = \mathcal{N}(\AA^T) = \mathcal{N}(\AA^\dagger)$.
            Therefore, $\AA^\dagger\xx = \AA^\dagger\yy$ and $\|\xx\|^2 = \|\yy\|^2 + \|\zz\|^2$.
            This implies
            \begin{equation*}
                \frac{\|\AA^\dagger\xx\|}{\|\xx\|}
                = \frac{\|\AA^\dagger\yy\|}{\sqrt{\|\yy\|^2 + \|\zz\|^2}},
            \end{equation*}
            which is maximized when $\zz=\zero$.
            Thus, 
            \begin{equation*}
                \|\AA^\dagger\| 
                = \max_{\yy\in\mathcal{R}(\AA)\backslash\{\zero\}} \frac{\|\AA^\dagger\yy\|}{\|\yy\|}.
            \end{equation*}
            As $\AA_2$ has full column rank then so does $\AA$. 
            Thus, $\AA^\dagger$ is the left inverse of $\AA$ and $\AA\vv = \zero$ only when $\vv = \zero$. 
            Therefore, 
            \begin{equation*}
                \|\AA^\dagger\| 
                = \max_{\vv\neq\zero} \frac{\|\AA^\dagger\AA\vv\|}{\|\AA\vv\|} 
                = \max_{\vv\neq\zero} \frac{\|\vv\|}{\|\AA\vv\|} 
                = \bigg(\min_{\vv\neq\zero} \frac{\|\AA\vv\|}{\|\vv\|}\bigg)^{-1} 
                = \bigg(\min_{\|\vv\|=1} \|\AA\vv\|\bigg)^{-1}.
            \end{equation*}
            This implies
            \begin{equation*}
                \|\AA^\dagger\|^{-1}
                = \min_{\|\vv\|=1} \|\AA\vv\|
                = \min_{\|\vv\|=1} \Bigg\|\begin{bmatrix}\AA_1\vv \\ \AA_2\vv\end{bmatrix}\Bigg\| 
                \geq \min_{\|\vv\|=1} \|\AA_2\vv\|
                = \|\AA_2^{-1}\|^{-1}. \qedhere
            \end{equation*}
        \end{proof}
        
        The inequality~\eqref{eq: H tilde dagger norm bound} follows immediately from Lemma~\ref{lem: pseudoinverse norm bound}. 
        With this, we now give a proof of Theorem~\ref{theorem: Case 2}.
        
        \begin{proof}[Proof of Theorem~\ref{theorem: Case 2}]
            Suppose iteration $t$ is in \labelcref{Case 2}.
            For this iteration, we use the update direction $\pp_t = \sum_{i=1}^{m}\pp_{t,i}/m$, with $\pp_{t,i}=-\tilde{\HH}_{t,i}^{\dagger}\tilde{\bgg}_t$.
            It follows from \eqref{eq: H tilde dagger norm bound} that
            \begin{equation*}
                \|\pp_t\|
                \leq \frac{1}{m}\sum_{i=1}^{m} \|\pp_{t,i}\|
                \leq \frac{1}{m}\sum_{i=1}^{m} \frac{1}{\phi} \|\tilde{\bgg}_t\|
                = \frac{1}{\phi} \|\tilde{\bgg}_t\|
                = \frac{1}{\phi} \|\bgg_t\|,
            \end{equation*}
            where $\phi$ is as in Algorithm~\ref{alg: Our Method}.
            From this and an analogous argument to that in the proof of Theorem~\ref{theorem: Case 1}, for all $\alpha\in(0,\tau_2]$, where $\tau_2$ is as in \eqref{eq: Case 2 Convergence}, we have
            \begin{equation*}
                \|\grad f(\ww_t+\alpha\pp_t)\|^2
                \leq \|\bgg_t\|^2 + 2\alpha\rho\langle\pp_t,\HH_t\bgg_t\rangle.
            \end{equation*}
            Therefore, line-search \eqref{eq: Armijo-type line seach} will pass for some $\alpha_t \geq \tau_2$ and $\| \bgg_{t+1} \|^2 \leq (1-2\tau_2\rho\theta) \|\bgg_t\|^2$.
            Moreover, from an analogous argument to that in the proof of Theorem~\ref{theorem: Case 1}, we have $\theta\leq\sqrt{L(\ww_0)}/\phi$, which implies $0\leq 1-2\tau_2\rho\theta <1$.
        \end{proof}

    \subsection{Theorem \ref{theorem: Case 2 Inexact}}
        \begin{proof}[Proof of Theorem~\ref{theorem: Case 2 Inexact}]
            Suppose iteration $t$ is in \labelcref{Case 2*}.
            For this iteration, we use the update direction $\pp_t = \sum_{i=1}^{m}\pp_{t,i}$, where $\pp_{t,i}=-\vv_{t,i}^{(2)}$ and $\vv_{t,i}^{(2)}$ is an approximation of $\tilde{\HH}_{t,i}^\dagger\tilde{\bgg}_t$ satisfying the inexactness condition in \eqref{eq: Case 2 inexactness condition}. 
            It follows from \eqref{eq: H tilde dagger norm bound} and \eqref{eq: Case 2 inexactness condition corollary} that
            \begin{equation*}
                \|\pp_t\|
                \leq \frac{1}{m}\sum_{i=1}^{m}\|\vv_{t,i}^{(2)}\|
                \leq \frac{1}{m}\sum_{i=1}^{m}
                \bigg(1 + \varepsilon_i^{(2)} \frac{K_i^2+\phi^2}{\phi^2}\bigg) \|\tilde{\HH}_{t,i}^\dagger\tilde{\bgg}_t\|
                \leq c \|\bgg_t\|,
            \end{equation*}
            where $c$ is as in \eqref{eq: Case 2 Inexact Convergence}.
            From this and an analogous argument to that in the proof of Theorem~\ref{theorem: Case 1}, for all $\alpha\in(0,\tilde{\tau}_2]$, where $\tilde{\tau}_2$ is as in \eqref{eq: Case 2 Inexact Convergence}, we have
            \begin{equation*}
                \big\|\grad f(\ww_t+\alpha\pp_t)\big\|^2
                \leq \|\bgg_t\|^2 + 2\alpha\rho\langle\pp_t,\HH_t\bgg_t\rangle.
            \end{equation*}
            Therefore, line-search \eqref{eq: Armijo-type line seach} will pass for some $\alpha_t \geq \tilde{\tau}_2$ and $\| \bgg_{t+1} \|^2 \leq (1-2\tilde{\tau}_2\rho\theta) \|\bgg_t\|^2$.
            Moreover, from an analogous argument to that in the proof of Theorem~\ref{theorem: Case 1}, we have $\theta\leq c\sqrt{L(\ww_0)}$, which implies $0\leq 1-2\tilde{\tau}_2\rho\theta <1$.
        \end{proof}

    \subsection{Theorem \ref{theorem: Case 3}}\label{section: proof of case 3}
        \begin{proof}[Proof of Theorem~\ref{theorem: Case 3}]
            Suppose iteration $t$ is in \labelcref{Case 3}.
            Recall that for this iteration, each worker $i\in\mathcal{I}_t\neq\emptyset$, as defined in \eqref{eq: set of Case 3 iteration indices}, computes
            \begin{equation*}
                \pp_{t,i} = -\tilde{\HH}_{t,i}^\dagger\tilde{\bgg}_t - \lambda_{t,i}(\tilde{\HH}_{t,i}^T\tilde{\HH}_{t,i})^{-1}\HH_t\bgg_t, 
                \quad\text{where}\quad
                \lambda_{t,i} = \frac{- \langle \tilde{\HH}_{t,i}^\dagger\tilde{\bgg}_t, \HH_t\bgg_t \rangle + \theta\|\bgg_t\|^2}
                    {\big\langle (\tilde{\HH}_{t,i}^T\tilde{\HH}_{t,i})^{-1}\HH_t\bgg_t,\HH_t\bgg_t \big\rangle}
                > 0.
            \end{equation*}
            The term $\lambda_{t,i}$ is both well-defined and positive by Assumptions~\ref{assumption: Case 3 Gamma} and \ref{assumption: GHNSP}, lines 4 and 5 of Algorithm~\ref{alg: Our Method}, and the definition of $\mathcal{I}_t$.
            It follows from
            $\big\langle (\tilde{\HH}_{t,i}^T\tilde{\HH}_{t,i})^{-1}\HH_t\bgg_t,\HH_t\bgg_t \big\rangle = \big\|(\tilde{\HH}_{t,i}^T)^\dagger\HH_t\bgg_t\big\|^2$ and inequality~\eqref{eq: H tilde dagger norm bound} that for all $i\in\mathcal{I}_t$ we have
            \begin{align*}
                \lambda_{t,i} \big\|(\tilde{\HH}_{t,i}^T\tilde{\HH}_{t,i})^{-1}\HH_t\bgg_t\big\| 
                &= \big({-}\langle \tilde{\HH}_{t,i}^\dagger\tilde{\bgg}_t, \HH_t\bgg_t \rangle + \theta\|\bgg_t\|^2\big) 
                    \frac{\big\|(\tilde{\HH}_{t,i}^T\tilde{\HH}_{t,i})^{-1}\HH_t\bgg_t\big\|}
                        {\big\|(\tilde{\HH}_{t,i}^T)^\dagger\HH_t\bgg_t\big\|^2} \\
                &= \Bigg(
                        \frac{{-}\langle \tilde{\HH}_{t,i}^\dagger\tilde{\bgg}_t, \HH_t\bgg_t \rangle + \theta\|\bgg_t\|^2}
                            {\big\|(\tilde{\HH}_{t,i}^T)^\dagger\HH_t\bgg_t\big\|}
                    \Bigg)
                    \Bigg(
                        \frac{\big\|\tilde{\HH}_{t,i}^\dagger\big((\tilde{\HH}_{t,i}^T)^\dagger\HH_t\bgg_t\big)\big\|}
                            {\big\|(\tilde{\HH}_{t,i}^T)^\dagger\HH_t\bgg_t\big\|}
                    \Bigg) \\
                &\leq \frac{1}{\phi} 
                    \Bigg(
                        \frac{{-}\langle \tilde{\HH}_{t,i}^\dagger\tilde{\bgg}_t, \HH_t\bgg_t \rangle}
                            {\big\|(\tilde{\HH}_{t,i}^T)^\dagger\HH_t\bgg_t\big\|}
                        +
                        \frac{\theta\|\bgg_t\|^2}
                            {\big\|(\tilde{\HH}_{t,i}^T)^\dagger\HH_t\bgg_t\big\|}
                    \Bigg).
            \end{align*}
            Moreover, by \eqref{eq: case 3 lower bound}, for all $i\in\mathcal{I}_t$ we have
            \begin{align*}
                \lambda_{t,i}\big\|(\tilde{\HH}_{t,i}^T\tilde{\HH}_{t,i})^\dagger\HH_t\bgg_t\big\| 
                &\leq \frac{1}{\phi} 
                    \Bigg(
                        \sqrt{\frac{K_i^2+\phi^2}{\phi^2}}\|\bgg_t\|
                        +
                        \frac{\theta}{\gamma} \sqrt{\frac{K_i^2+\phi^2}{\nu}}\|\bgg_t\|
                    \Bigg) \\
                &= \bigg( \frac{1}{\phi} + \frac{\theta}{\gamma\sqrt{\nu}} \bigg) \sqrt{\frac{K_i^2+\phi^2}{\phi^2}} \|\bgg_t\|.
            \end{align*}
            Therefore, for all $i\in\mathcal{I}_t$ we have
            \begin{equation*}
                \|\pp_{t,i}\|
                \leq \|\tilde{\HH}_{t,i}^\dagger\tilde{\bgg}_t\| + \lambda_{t,i} \big\|(\tilde{\HH}_{t,i}^T\tilde{\HH}_{t,i})^{-1}\HH_t\bgg_t\big\|
                \leq \frac{1}{\phi}\|\bgg_t\| + \bigg( \frac{1}{\phi} + \frac{\theta}{\gamma\sqrt{\nu}} \bigg) \sqrt{\frac{K_i^2+\phi^2}{\phi^2}} \|\bgg_t\|.
            \end{equation*}
            This, and how $\pp_{t,i} = -\tilde{\HH}_{t,i}^\dagger\tilde{\bgg}_t$ for $i\notin\mathcal{I}_t$, implies
            \begin{equation*}
                \|\pp_t\| 
                \leq \frac{1}{m}\bigg(\sum_{i\notin \mathcal{I}_t}\|\pp_{t,i}\| + \sum_{i\in \mathcal{I}_t}\|\pp_{t,i}\|\bigg)
                \leq c \|\bgg_t\|,
            \end{equation*}
            where $c$ is as in \eqref{eq: Case 3 Convergence}.
            From this and an analogous argument to that in the proof of Theorem~\ref{theorem: Case 1}, for $\alpha\in(0,\tau_3]$, where $\tau_3$ is as in \eqref{eq: Case 3 Convergence}, we have 
            \begin{equation*}
                \big\|\grad f(\ww_t+\alpha\pp_t)\big\|^2 \leq \|\bgg_t\|^2 + 2\alpha\rho\langle\pp_t,\HH_t\bgg_t\rangle.
            \end{equation*}
            Therefore, line-search \eqref{eq: Armijo-type line seach} will pass for some $\alpha_t \geq \tau_3$ and $\| \bgg_{t+1} \|^2 \leq (1-2\tau_3\rho\theta) \|\bgg_t\|^2$. 
            Moreover, from an analogous argument to that in the proof of Theorem~\ref{theorem: Case 1}, we have \hbox{$\theta\leq c\sqrt{L(\ww_0)}$}, which implies $0\leq 1-2\tau_3\rho\theta<1$. 
            Because $\gamma\leq\sqrt{L(\ww_0)}$ and $\nu\in(0,1]$, the inequality $\theta\leq c\sqrt{L(\ww_0)}$ holds for all $\theta>0$.
        \end{proof}

    \subsection{Lemma \ref{lem: necessary conditions}}\label{section: proof of necessary conditions lemma}
        \begin{proof}[Proof of Lemma \ref{lem: necessary conditions}]
            Suppose $|\mathcal{I}_t|<m$ for some iteration $t$, as in \eqref{eq: set of Case 3 iteration indices}.
            Therefore, for all $i\notin\mathcal{I}_t$, inequality \eqref{eq: H tilde dagger norm bound} implies
            \begin{equation*}
                \theta\|\bgg_t\|^2 
                \leq \langle\tilde{\HH}_{t,i}^\dagger\tilde{\bgg}_t,\HH_t\bgg_t\rangle
                \leq \|\bgg_t\| \cdot \big\|(\tilde{\HH}_{t,i}^T)^\dagger\HH_t\bgg_t\big\|
                \leq \frac{1}{\phi}\|\bgg_t\|\cdot\|\HH_t\bgg_t\|,
            \end{equation*}
            which gives $\theta\phi \leq \|\HH_t\bgg_t\|/\|\bgg_t\|$. 
        \end{proof}
        
    \subsection{Theorem \ref{theorem: Case 3 Inexact}}\label{section: proof of case 3 Inexact}
        \begin{proof}[Proof of Theorem~\ref{theorem: Case 3 Inexact}]
            Suppose iteration $t$ is in \labelcref{Case 3*}.
            Recall that for this iteration, each worker $i\in\mathcal{I}_t^*\neq\emptyset$, as defined in \eqref{eq: set of Case 3* iteration indices}, computes
            \begin{equation*}
                \pp_{t,i} 
                = - \vv_{t,i}^{(2)} - \lambda_{t,i}\vv_{t,i}^{(3)}, 
                \quad\text{where}\quad
                \lambda_{t,i} 
                = \frac{-\langle \vv_{t,i}^{(2)}, \HH_t\bgg_t \rangle + \theta\|\bgg_t\|^2}
                {\langle \vv_{t,i}^{(3)}, \HH_t\bgg_t \rangle} 
                > 0.
            \end{equation*}
            The term $\lambda_{t,i}$ is both well-defined and positive by Assumptions~\ref{assumption: Case 3 Gamma} and \ref{assumption: GHNSP}, lines 4 and 5 of Algorithm~\ref{alg: Our Method}, the definition of $\mathcal{I}_{t}^*$ and the condition in \eqref{eq: Case 3 inexactness condition}.
            The inexactness condition in \eqref{eq: Case 3 inexactness condition} implies
            \begin{align*}
                - \langle \vv_{t,i}^{(3)}, \HH_t\bgg_t \rangle 
                + \big\langle (\tilde{\HH}_{t,i}^T\tilde{\HH}_{t,i})^{-1}\HH_t\bgg_t,\HH_t\bgg_t \big\rangle
                &=
                - \big\langle  
                \tilde{\HH}_{t,i}^T\tilde{\HH}_{t,i}\vv_{t,i}^{(3)} - \HH_t\bgg_t,
                (\tilde{\HH}_{t,i}^T\tilde{\HH}_{t,i})^{-1}\HH_t\bgg_t
                \big\rangle \\
                &\leq \varepsilon_i^{(3)} \|\HH_t\bgg_t\| \big\|(\tilde{\HH}_{t,i}^T\tilde{\HH}_{t,i})^{-1}\HH_t\bgg_t\big\|.
            \end{align*}
            By Assumption~\ref{assumption: Lipschitz Continuity}, we have
            \begin{equation*}
                \varepsilon_i^{(3)} \|\HH_t\bgg_t\| \big\|(\tilde{\HH}_{t,i}^T\tilde{\HH}_{t,i})^{-1}\HH_t\bgg_t\big\|
                \leq 
                \varepsilon_i^{(3)} \sqrt{\frac{K_i^2+\phi^2}{\phi^2}} 
                \big\langle (\tilde{\HH}_{t,i}^T\tilde{\HH}_{t,i})^{-1}\HH_t\bgg_t,\HH_t\bgg_t \big\rangle.
            \end{equation*}
            Therefore,
            \begin{equation*}
                \langle \vv_{t,i}^{(3)}, \HH_t\bgg_t \rangle
                \geq 
                \Bigg( 1 - \varepsilon_i^{(3)} \sqrt{\frac{K_i^2+\phi^2}{\phi^2}} \Bigg)
                \big\langle (\tilde{\HH}_{t,i}^T\tilde{\HH}_{t,i})^{-1}\HH_t\bgg_t,\HH_t\bgg_t \big\rangle,
            \end{equation*}
            where the right-hand side is positive by Assumptions~\ref{assumption: Case 3 Gamma} and \ref{assumption: GHNSP}, lines 4 and 5 of Algorithm~\ref{alg: Our Method}, and how it is assumed that $\varepsilon_i^{(3)} < \sqrt{\phi^2/(K_i^2+\phi^2)}$.
            It follows from \eqref{eq: Case 2 inexactness condition corollary} and \eqref{eq: Case 3 inexactness condition corollary} that
            \begin{align*}
                \lambda_{t,i} \|\vv_{t,i}^{(3)}\|
                &\leq 
                \Bigg(\frac{1+\varepsilon_i^{(3)}(K_i^2+\phi^2)/\phi^2}{1-\varepsilon_i^{(3)}\sqrt{(K_i^2+\phi^2)/\phi^2}}\Bigg)
                \big({-\langle \vv_{t,i}^{(2)}, \HH_t\bgg_t \rangle} + \theta\|\bgg_t\|^2\big)
                \frac{\|(\tilde{\HH}_{t,i}^T\tilde{\HH}_{t,i})^{-1}\HH_t\bgg_t\big\|}{\big\|(\tilde{\HH}_{t,i}^T)^\dagger\HH_t\bgg_t\big\|^2} \\
                &\leq 
                \frac{1}{\phi} 
                \Bigg(\frac{1+\varepsilon_i^{(3)}(K_i^2+\phi^2)/\phi^2}{1-\varepsilon_i^{(3)}\sqrt{(K_i^2+\phi^2)/\phi^2}}\Bigg)
                \Bigg(
                \frac{\|\vv_{t,i}^{(2)}\| \|\HH_t\bgg_t\| + \theta\|\bgg_t\|^2}{\big\|(\tilde{\HH}_{t,i}^T)^\dagger\HH_t\bgg_t\big\|}
                \Bigg) \\
                &\leq
                \Bigg(\frac{1+\varepsilon_i^{(3)}(K_i^2+\phi^2)/\phi^2}{1-\varepsilon_i^{(3)}\sqrt{(K_i^2+\phi^2)/\phi^2}}\Bigg)
                \Bigg(
                \frac{1}{\phi} \bigg(1+\varepsilon_i^{(2)}\frac{K_i^2+\phi^2}{\phi^2}\bigg) + \frac{\theta}{\gamma\sqrt{\nu}}
                \Bigg) \sqrt{\frac{K_i^2+\phi^2}{\phi^2}} \|\bgg_t\|.
            \end{align*}
            This and \eqref{eq: Case 2 inexactness condition corollary}, and how $\pp_{t,i} = - \vv_{t,i}^{(2)}$ for $i\notin\mathcal{I}_t^*$, imply
            \begin{equation*}
                \|\pp_t\| 
                \leq \frac{1}{m}\bigg(\sum_{i\notin \mathcal{I}_t}\|\pp_{t,i}\| + \sum_{i\in \mathcal{I}_t}\|\pp_{t,i}\|\bigg)
                \leq \frac{1}{m}\bigg(\sum_{i=1}^{m} \|\vv_{t,i}^{(2)}\| 
                + \sum_{i\in \mathcal{I}_t} \lambda_{t,i} \|\vv_{t,i}^{(3)}\|\bigg) 
                \leq c \|\bgg_t\|,
            \end{equation*}
            where $c$ is as in \eqref{eq: Case 3 Inexact Convergence c}.
            From this and an analogous argument to that in the proof of Theorem~\ref{theorem: Case 1}, for all $\alpha\in(0,\tilde{\tau}_3]$, where $\tilde{\tau}_3$ is as in \eqref{eq: Case 3 Inexact Convergence tau}, we have
            \begin{equation*}
                \big\|\grad f(\ww_t+\alpha\pp_t)\big\|^2
                \leq \|\bgg_t\|^2 + 2\alpha\rho\langle\pp_t,\HH_t\bgg_t\rangle.
            \end{equation*}
            Therefore, line-search \eqref{eq: Armijo-type line seach} will pass for some $\alpha_t \geq \tilde{\tau}_3$ and $\| \bgg_{t+1} \|^2 \leq (1-2\tilde{\tau}_3\rho\theta) \|\bgg_t\|^2$.
            Moreover, from an analogous argument to that in the proof of Theorem~\ref{theorem: Case 1}, we have \hbox{$\theta\leq c\sqrt{L(\ww_0)}$}, which implies $0\leq 1-2\tilde{\tau}_3\rho\theta<1$. 
            Because $\gamma\leq\sqrt{L(\ww_0)}$ and $\nu\in(0,1]$, the inequality $\theta\leq c\sqrt{L(\ww_0)}$ holds for all $\theta>0$.
        \end{proof}
        
\end{document}